\documentclass[a4paper,10pt]{article}
\usepackage[top=2.54cm,bottom=2.0cm,left=2.0cm,right=2.54cm, includeheadfoot]{geometry}

\usepackage[T1]{fontenc}
\usepackage[utf8]{inputenc}
\usepackage[english]{babel}
\usepackage{enumerate}
\usepackage{multirow,booktabs}
\usepackage[table]{xcolor}
\usepackage{fullpage}
\usepackage{lastpage}
\usepackage{indentfirst}

\usepackage{footnote}

\usepackage{amsmath,amsfonts,amssymb,amscd,amsthm}

\usepackage{bm}

\usepackage[all,2cell]{xy} \UseAllTwocells \SilentMatrices

\usepackage{hyperref}

\usepackage{subfiles}

\usepackage{multicol}

\usepackage{showlabels}

\usepackage{lineno}

\newtheorem{teo}{Theorem}[section]
\newtheorem{lem}[teo]{Lemma} 

\newtheorem{prop}[teo]{Proposition} 

\newtheorem{defn}[teo]{Definition} 
\newtheorem{ex}[teo]{Example}

\newtheorem*{claim*}{Claim}
\newtheorem{op}[teo]{Question}
\newtheorem{rem}[teo]{Remark}

\newcommand{\kmar}[1]{{\color{blue}{#1}}}

\begin{document}

\title{Linear Systems, Matrices and Vector Spaces over Superfields}

\author{Kaique Matias de Andrade Roberto \& Hugo Rafael de Oliveira Ribeiro \& \\  
Hugo Luiz Mariano \&  Kaique Ribeiro Prates Santos}

\maketitle

\begin{abstract}
    
   Motivated by some recent developments in abstract theories of quadratic forms, we start to develop in this work an  expansion of Linear Algebra  to multivalued structures (a multialgebraic structure is essentially an algebraic structure but endowed with some multivalued operations). 
    We introduce and study matrices and determinants over a commutative superrings   
  (roughly, a ring where the sum and product are multivalued)  and study linear systems and vector spaces over superfields. As an application, we obtain a fundamental result to the development of a  theory of algebraic extensions of superfields.
\end{abstract}

\section*{Introduction}

 Motivated by some recent developments in abstract theories of quadratic forms, we start in \cite{roberto2022ACmultifields2} a program of study of algebraic extensions of superfield and,  connect to this, we also start a development of an  expansion of Linear Algebra over to multivalued structures:  present the latter part  is the goal of the present work.

The concept of multialgebraic structure -- an ``algebraic like'' structure but endowed with some multiple valued operations -- has been studied since the 1930's; in particular, the concept of hyperrings was introduced by Krasner in the 1950's: this is essentially a ring with a multivalued operation of sum. In the same vein multiring (\cite{marshall2006real} is essentially a ring with multivalued sum but satisfying a lax distributive law and a superring (\cite{ameri2019superring})  is, roughly, a ring where the sum and product are multivalued. 
Some general algebraic study has been made on multialgebras: see for instance \cite{golzio2018brief} and \cite{pelea2006multialgebras}.

Since the middle of the 2000s decade, the notion of multiring has obtained more attention: a multiring is a lax hyperring, satisfying a weak distributive law, but  hyperfields and multifields coincide. Multirings   have been studied for  applications in many areas: in abstract quadratic forms theory (\cite{marshall2006real}, \cite{worytkiewiczwitt2020witt}, \cite{roberto2021quadratic}), tropical geometry (\cite{viro2010hyperfields}, \cite{jun2015algebraic}), algebraic geometry ((\cite{jun2021geometry}, \cite{baker2021descartes}), valuation theory (\cite{jun2018valuations}), Hopf algebras (\cite{eppolito2020hopf}), etc (\cite{baker2021structure}, \cite{ameri2020advanced}, \cite{ameri2017multiplicative}, \cite{bowler2021classification}). 
A more detailed account of variants of the concept of polynomials over hyperrings is even more recent: see \cite{jun2015algebraic}, \cite{ameri2019superring}, \cite{baker2021descartes}, and \cite{roberto2021superrings}.

We briefly describe the structure of the present work:
 In Section 1 we present preliminaries on multirings, supperings and superring of polynomials. In Section 2 we introduce and study matrices and determinants over commutative superrings. Section 3   contains a study of linear systems over superfields. In Section 4   we provide an application of the previously developed material: we obtain a simple algebraic extension of a superfield by adding a root of an irreducible polynomial (Theorem \ref{teohell2}), a fundamental result to the development of a  theory of algebraic extensions of superfields.
In Section 5 we gave the first step to establish a theory of vector spaces over superfields. We finish the work in Section 6, presenting some perspectives.

 



\section{Preliminaries}

\subsection{Multi-structures}


\begin{defn}[Adapted from definition 1.1 in \cite{marshall2006real}]\label{defn:multimonoid}
 A \textbf{multigroup} is a first-order structure  $(G,\cdot,r,1)$ where $G$ is a non-empty set, 
$r:G\rightarrow G$ is a function, $1$ is an element of $G$, $\cdot \subseteq G\times G\times G$ is a ternary relation 
(that will play the role of binary multioperation, we denote $d\in a\cdot b$ for $(a,b,d)\in\cdot$) such that for all 
$a,b,c,d\in G$:
 \begin{description}
 \item [M1 - ] If $c\in a\cdot b$ then $a\in c\cdot(r(b))\wedge b\in(r(a))\cdot c$. We write $a\cdot b^{-1}$ to simplify 
$a\cdot(r(b))$.
 \item [M2 - ] $b\in a\cdot1$ iff $a=b$.
 \item [M3 - ] If $\exists\,x(x\in a\cdot b\wedge t\in x\cdot c)$ then
$\exists\,y(y\in b\cdot c\wedge t\in a\cdot y)$.
 \item [M4 - ] $c\in a\cdot b$ iff $c\in b\cdot a$.
\end{description}
The structure $(G,\cdot,r,1)$ is said to be \textbf{commutative (or abelian)} if satisfy for all $a,b,c\in G$ the condition
 \begin{description}
 \item [M4 - ] $c\in a\cdot b$ iff $c\in b\cdot a$.
\end{description}
The structure $(G,\cdot,1)$ is a \textbf{commutative multimonoid (with unity)} if satisfy M3 and M4 and the condition 
$a\in1\cdot a$ for all $a\in G$.
\end{defn}

\begin{defn}[Adapted from Definition 2.1 in \cite{marshall2006real}]\label{defn:multiring}
 A multiring is a sextuple $(R,+,\cdot,-,0,1)$ where $R$ is a non-empty set, $+:R\times 
R\rightarrow\mathcal P(R)\setminus\{\emptyset\}$,
 $\cdot:R\times R\rightarrow R$
 and $-:R\rightarrow R$ are functions, $0$ and $1$ are elements of $R$ satisfying:
 \begin{enumerate}[i -]
  \item $(R,+,-,0)$ is a commutative multigroup;
  \item $(R,\cdot,1)$ is a commutative monoid;
  \item $a.0=0$ for all $a\in R$;
  \item If $c\in a+b$, then $c.d\in a.d+b.d$. Or equivalently, $(a+b).d\subseteq a.d+b.d$.
 \end{enumerate}

Note that if $a \in R$, then $0 = 0.a \in (1+ (-1)).a \subseteq 1.a + (-1).a$, thus $(-1). a = -a$.
 
 $R$ is said to be an hyperring if for $a,b,c \in R$, $a(b+c) = ab + ac$. 
 
 A multiring (respectively, a hyperring) $R$ is said to be a multidomain (hyperdomain) if it hasn't zero divisors. A multiring 
$R$ will be a 
multifield if every non-zero element of $R$ has 
multiplicative inverse; note that hyperfields and multifields coincide.
\end{defn}

\begin{ex}\label{ex:1.3}
$ $
 \begin{enumerate}[a -]
  \item Suppose that $(G,\cdot,1)$ is a group. Defining $a \ast b = \{a \cdot b\}$ and $r(g)=g^{-1}$, 
we have that $(G,\ast,r,1)$ is a multigroup. In this way, every ring, domain and field is a multiring, 
multidomain and multifield, respectively.
  
  \item Let $K=\{0,1\}$ with the usual product and the sum defined by relations $x+0=0+x=x$, $x\in K$ and 
$1+1=\{0,1\}$. This is a multifield  called Krasner's multifield \cite{jun2015algebraic}. 
  
  \item $Q_2=\{-1,0,1\}$ is ``signal'' multifield with the usual product (in $\mathbb Z$) and the multivalued sum defined by 
relations
  $$\begin{cases}
     0+x=x+0=x,\,\mbox{for every }x\in Q_2 \\
     1+1=1,\,(-1)+(-1)=-1 \\
     1+(-1)=(-1)+1=\{-1,0,1\}
    \end{cases}
  $$
  \end{enumerate}
\end{ex}

\begin{ex}[Tropical Hyperfield \cite{viro2010hyperfields}]\label{tropical-ex}

For a fixed totally ordered abelian group $(G, +, -, 0, \leq)$ we can construct a  {\em tropical multifield} $T_G = (G \cup \{\infty\}, \otimes, \odot, \ominus, 0, 1) $ where:
\begin{multicols}{2}
\begin{enumerate}[i -]
    \item $\forall g \in G, g < \infty$;
    \item $g \otimes h := g+h$;
    \item $0 := \infty$;
    \item $1 := 0$;
    \item $\forall g \in G$, $g^{-1} = -g$;
    \item if $g \neq h$, $g \oplus h = \{min\{g,h\}\}$;
    \item $g \oplus g = \{h \in G \cup \{\infty\} : g \leq h\}$; 
    \item $\ominus g = g$.
\end{enumerate}
\end{multicols}
\end{ex}

In the sequence, we provide examples that generalize the previous ones.

\begin{ex}[H-multifield, Example 2.8 in \cite{ribeiro2016functorial}]\label{H-multi}
Let $p\ge1$ be a prime integer and $H_p:=\{0,1,...,p-1\} \subseteq \mathbb{N}$. Now, define the binary multioperation and 
operation in $H_p$ as 
follow:
\begin{align*}
 a+b&=
 \begin{cases}H_p\mbox{ if }a=b,\,a,b\ne0 \\ \{a,b\} \mbox{ if }a\ne b,\,a,b\ne0 \\ \{a\} \mbox{ if }b=0 \\ \{b\}\mbox{ if 
}a=0 \end{cases} \\
 a\cdot b&=k\mbox{ where }0\le k<p\mbox{ and }k\equiv ab\mbox{ mod p}.
\end{align*}
$(H_p,+,\cdot,-, 0,1)$ is a hyperfield such that for all $a\in H_p$, $-a=a$. In fact, these $H_p$ are a kind of generalization of $K$, in the sense that $H_2=K$.
\end{ex}

\begin{ex}[Kaleidoscope, Example 2.7 in \cite{ribeiro2016functorial}]\label{kaleid}
 Let $n\in\mathbb{N}$ and define 
 $$X_n=\{-n,...,0,...,n\} \subseteq \mathbb{Z}.$$ 
 We define the \textbf{$n$-kaleidoscope multiring} by 
$(X_n,+,\cdot,-, 0,1)$, where $- : X_n \to X_n$ is restriction of the  opposite map in $\mathbb{Z}$,  $+:X_n\times 
X_n\rightarrow\mathcal{P}(X_n)\setminus\{\emptyset\}$ is given by the rules:
 $$a+b=\begin{cases}
    \{a\},\,\mbox{ if }\,b\ne-a\mbox{ and }|b|\le|a| \\
    \{b\},\,\mbox{ if }\,b\ne-a\mbox{ and }|a|\le|b| \\
    \{-a,...,0,...,a\}\mbox{ if }b=-a
   \end{cases},$$
and $\cdot:X_n\times X_n\rightarrow X_n$ is given by the rules:
 $$a\cdot b=\begin{cases}
    \mbox{sgn}(ab)\max\{|a|,|b|\}\mbox{ if }a,b\ne0 \\
    0\mbox{ if }a=0\mbox{ or }b=0
   \end{cases}.$$
With the above rules we have that $(X_n,+,\cdot, -, 0,1)$ is a multiring which is not a hyperring for $n\ge2$ because $$n(1-1)=b\cdot\{-1,0,1\}=\{-n,0,n\}$$
and $n-n=X_n$. Note that $X_0=\{0\}$ and $X_1=\{-1,0,1\} =  Q_2$. 
\end{ex}

\begin{ex}[Multigroup of a Linear Order, 3.4 of \cite{viro2010hyperfields}]
Let $(\Gamma,\cdot,1,\le)$ be an ordered abelian group. We have an associated hyperfield structure $(\Gamma\cup\{0\},+,-,\cdot,0,1)$ with the rules $-a:=a$, $a\cdot0=0\cdot a:=0$ and
$$a+b:=\begin{cases}a \mbox{ if }a<b \\ b\mbox{ if }b<a\\
[0,a]\mbox{ if }a=b\end{cases}$$
Here we use the convention $0\le a$ for all $a\in\Gamma$. 
\end{ex} 

 Now, we treat about morphisms:

\begin{defn}\label{defn:morphism}
 Let $A$ and $B$ multirings. A map $f:A\rightarrow B$ is a morphism if for all $a,b,c\in A$:
 \begin{multicols}{2}
 \begin{enumerate}[i -]
  \item $f(1)=1$ and $f(0)=0$;
  \item $f(-a)=-f(a)$;
  \item $f(ab)=f(a)f(b)$;
  \item $c\in a+b\Rightarrow f(c)\in f(a)+f(b)$.
 \end{enumerate}
 \end{multicols}
  A morphism $f$ is \textbf{a full morphism} if for all $a,b\in A$, $$f(a+b)=f(a)+f(b)\mbox{ and }f(a\cdot b)=f(a)\cdot f(b).$$
\end{defn}

\begin{ex}
$ $
\begin{enumerate}[i -]  
\item The prime ideals  of a commutative ring (its Zariski spectrum) are classified by equivalence classes of morphisms into  algebraically closed  fields,  but they can be {\em uniformly classified} by a multiring morphism into the Krasner multifield $K = \{0,1\}$.

\item The orderings of a commutative ring (its real spectrum) are classified by classes of equivalence of ring homomorphims into  real closed fields, but they can be {\em uniformly classified} by a multiring morphism into the signal multifield  $Q_2 = \{-1,0,1\}$. 

\item A Krull valuation on a commutative ring with group of values $(G, +, -, 0, \leq)$  is just a morphism into the multifield $T_G=G \cup\{\infty\}$.

\end{enumerate}

\end{ex}

    
    
    

 The concept of superring first appears in  (\cite{ameri2019superring}). There are many important advances and results in multiring/hyperring theory, and for instance, we recommend for example, the following papers: \cite{al2019some}, \cite{ameri2017multiplicative}, \cite{ameri2019superring}, \cite{ameri2020advanced}, \cite{massouros1985theory}, \cite{nakassis1988recent}, \cite{massouros1999homomorphic}, \cite{massouros2009join}.
 
\begin{defn}[Definition 5 in \textup{\cite{ameri2019superring}}]
 A superring is a structure $(S,+,\cdot, -, 0,1)$ such that:
 \begin{enumerate}[i -]
  \item $(S,+, -, 0)$ is a commutative multigroup. 
  \item $(S,\cdot,1)$ is a multimonoid. 
  \item $0$ is an absorbing element: $a\cdot0= \{0\} = 0 \cdot a$, for all $a\in S$.
  \item The weak/semi distributive law holds: 
   if $d\in c.(a+b)$ and $e\in(a+b)c$ then $d\in ca+cb$ and $e\in ac+bc$, for all $a,b,c,d,e\in S$.  
  \item  The rule of signals holds:  $-(ab)=(-a)b=a(-b)$, for all $a,b\in S$.
 \end{enumerate}
 A superring is commutative if $(S,\cdot,1)$ is commutative. A superdomain is a non-trivial superring without zero-divisors in this new context, i.e. whenever
 $$0\in a\cdot b \mbox{ iff }a=0 \mbox{ or } b=0$$
 A quasi-superfield is a non-trivial superring such that every nonzero element is invertible in this new context\footnote{
 For a quasi-superfield $F$, we \textbf{are not imposing} that $(S\setminus\{0\},\cdot,1)$ will be a commutative multigroup, i.e, that if $d\in a\cdot b$ then $b^{-1}\in a\cdot d^{-1}$.}, i.e. whenever
 $$\mbox{ For all }a \neq 0 \mbox{ exists }b\mbox{ such that }1\in a\cdot b.$$
 A superfield is a quasi-superfield which is also a superdomain. A superring is full if for all $a,b,c,d\in S$, $d\in c\cdot(a+b)$ iff $d\in ca+cb$.
\end{defn}

From now on, all superrings will be commutative (with exceptions sinalized).


Now we treat about morphisms.

\begin{defn}
 Let $A$ and $B$ superrings. A map $f:A\rightarrow B$ is a morphism if for all $a,b,c\in A$:
 \begin{multicols}{2}
  \begin{enumerate}[i -]
  \item $f(0)=0$;
  \item $f(1)=1$;
  \item $f(-a)=-f(a)$;
  \item $c\in a+b\Rightarrow f(c)\in f(a)+f(b)$;
  \item $c\in a\cdot b\Rightarrow f(c)\in f(a)\cdot f(b)$.
 \end{enumerate} 
 \end{multicols}
 A morphism $f$ is \textbf{a full morphism} if for all $a,b\in A$, $$f(a+b)=f(a)+f(b)\mbox{ and }f(a\cdot b)=f(a)+f(b).$$
\end{defn}


From now on, we use the following conventions: Let $(R,+,\cdot, -, 0,1)$ be a superring, $p\in\mathbb N$ and consider a $p$-tuple $\vec a=(a_0,a_1, ..., a_{p-1})$. We define the finite sum by:
\begin{align*}
 x\in\sum_{i<0}a_i&\mbox{ iff }x=0, \\
 x\in\sum_{i<p}a_i&\mbox{ iff }x\in y+a_{p-1}\mbox{ for some }y\in\sum_{i<p-1}a_i, \text{if} \ p \geq 1.
\end{align*}
and the finite product by:
\begin{align*}
 x\in\prod_{i<0}a_i&\mbox{ iff }x=1, \\
 x\in\prod_{i<p}a_i&\mbox{ iff }x\in y\cdot a_{p-1}\mbox{ for some }y\in\prod_{i<p-1}a_i, \text{if} \ p \geq 1.
\end{align*}

Thus, if $(\vec{a}_0, \vec{a}_1,...,\vec{a}_{p-1})$ is a $p$-tuple of tuples $\vec{a}_i = (a_{i0}, a_{i1},..., a_{i{m_i}})$, 
then we have the finite sum of finite products:
\begin{align*}
 x\in\sum_{i<0}\prod_{j<{m_i}}a_{ij}&\mbox{ iff }x=0, \\
 x\in\sum_{i<p}\prod_{j<{m_i}}a_{ij}&\mbox{ iff }x\in y+z\mbox{ for some }y\in
 \sum_{i<{p-1}}\prod_{j<{m_i}}a_{ij}
\mbox{ and } z\in\prod_{j<m_{p-1}}a_{{p-1},j}, \text{if} \ p \geq 1.
\end{align*}

\begin{lem}[Basic Facts]\label{lembasic1}
    Let $A$ be a superring.
    \begin{enumerate}[a -]
        \item For all $n\in\mathbb N$ and all $a_1,...,a_n\in A$, the sum $a_1+...+a_n$ and product $a_1\cdot...\cdot  a_n$ does not depends on the order of the entries.
        \item If $A$ is a full superdomain, then $ax=ay$ for some $a\ne0$ imply $x=y$.
        \item If $A$ is full, then for all $d,a_1,...,a_n\in A$
        $$d(a_1+...+a_n)=da_1+...+da_n.$$
        \item Suppose $A$ is a full superdomain and let $a\in A\setminus\{0\}$. If $1\in (a\cdot b)\cap(a\cdot c)$ then $b=c$.
        \item (Newton's Binom Formula) For $n\ge1$ and $X\subseteq A$ denote
            $$nX:=\sum^n_{i=1}X.$$
            Then for $A,B\subseteq A$,
            $$(A+B)^n\subseteq\sum^n_{j=0}\binom{n}{j}A^jB^{n-j}.$$
    \end{enumerate}
\end{lem}

\begin{lem}[Facts about full morphisms of superrings]\label{factstrong2}
Let $f:A\rightarrow B$ be a full morphism of superrings. Then
\begin{enumerate}[a -]
    \item For all $a_1,...,a_n\in A$,
    $$f(a_1+...+a_n)=f(a_1)+...+f(a_n).$$
    \item For all $a_1,...,a_n,b_1,...,b_n\in A$,
    $$f[(a_1+b_1)(a_2+b_2)...(a_n+b_n)]=(f(a_1)+f(b_1))(f(a_2)+f(b_2))...(f(a_n)+f(b_n)).$$
    \item For all $c_1,...,c_n,d_1,...,d_n\in A$,
    $$f(c_1d_1+c_2d_2+...+c_nd_n)=f(c_1)f(d_1)+f(c_2)f(d_2)+...+f(c_n)f(d_n).$$
    
    \item For all $a_0,...,a_n,\alpha\in A$,
    $$f(a_0+a_1\alpha+...+a_n\alpha^n)=f(a_0)+f(a_1)f(\alpha)+...+f(a_n)f(\alpha)^n.$$
    
    \item Let $A_1,A_2,A_3$ be superrings with injective morphisms (embeddings) $i_{12}:A_1\rightarrow A_2$, $i_{13}:A_1\rightarrow A_3$ and $i_{23}:A_2\rightarrow A_3$. 
    $$\xymatrix@!=4pc{A_1\ar[r]^{i_{12}}\ar[dr]_{i_{13}} & A_2\ar[d]^{i_{23}} \\ & A_3}$$
    Suppose that $i_{13}=i_{23}\circ i_{12}$ is a full embedding. If $i_{23}$ is a full embedding then $i_{12}$ is a full embedding.
\end{enumerate}
\end{lem}

\begin{defn}\label{char}
$ $
 \begin{enumerate}[i -]
  \item The \textbf{characteristic} of a superring is the smaller integer $n \geq 1$ such that
  $$0\in\sum_{i<n}1,$$
  otherwise the characteristic is zero. For full superdomains, this is equivalent to say that $n$ is the smaller integer such that
  $$\mbox{For all }a,\,0\in\sum_{i<n}a.$$
 
  \item An \textbf{ideal} of a superring $A$ is a non-empty subset $\mathfrak{a}$ of $A$ such that $\mathfrak{a}+\mathfrak{a}\subseteq\mathfrak{a}$ and $A\mathfrak{a}\subseteq\mathfrak{a}$. We denote
  $$\mathfrak I(A)=\{I\subseteq A:I\mbox{ is an ideal}\}.$$
  
  \item  Let $S$ be a subset of a superring $A$. We define the \textbf{ideal generated by} $S$ as 
  $$\langle S\rangle:=\bigcap\{\mathfrak{a}\subseteq A\mbox{ ideal}:S\subseteq\mathfrak{a}\}.$$
  If $S=\{a_1,...,a_n\}$, we easily check that
  $$\langle a_1,...,a_n\rangle=\sum Aa_1+...+\sum Aa_n,\,\mbox{where }\sum Aa=\bigcup\limits_{n\ge1}\{\underbrace{Aa+...+Aa}_{n\mbox{ times}}\}.$$ 
  Note that if $A$ is a full superring, then $\sum Aa=Aa$.
  
  \item An ideal $\mathfrak{p}$ of $A$ is said to be \textbf{prime} if $1\notin\mathfrak{p}$ and $ab\subseteq\mathfrak{p}\Rightarrow a\in\mathfrak{p}$ or $b\in\mathfrak{p}$. We denote 
  $$\mbox{Spec}(A)=\{\mathfrak{p}\subseteq A:\mathfrak{p}\mbox{ is a prime ideal}\}.$$
  
  \item An ideal $\mathfrak{p}$ of $A$ is said to be \textbf{strongly prime} if $1\notin\mathfrak{p}$ and $ab\cap\mathfrak{p}\ne\emptyset\Rightarrow a\in\mathfrak{p}$ or $b\in\mathfrak{p}$. We denote 
  $$\mbox{Spec}_s(A)=\{\mathfrak{p}\subseteq A:\mathfrak{p}\mbox{ is a strongly prime ideal}\}.$$
  Note that every strongly prime ideal is prime.
  
  \item An ideal $\mathfrak{m}$ is maximal if it is proper and for all ideals $\mathfrak{a}$ with $\mathfrak{m}\subseteq\mathfrak{a}\subseteq A$ then $\mathfrak{a}=\mathfrak{m}$ or $\mathfrak{a}=A$.
  
  \item For an ideal $I\subseteq A$, we define operations in the quotient $A/I=\{x+I:x\in A\}=\{\overline x:x\in A\}$, by the rules
  \begin{align*}
      \overline x+\overline y&=\{\overline z:z\in x+y\}\\
      \overline x\cdot\overline y&=\{\overline z:z\in xy\}
  \end{align*}
    for all $\overline x,\overline y\in A/I$.
 \end{enumerate}
\end{defn}

\begin{rem}
$ $
 \begin{enumerate}[a -]
     \item If $A$ is a multiring, then every prime ideal is strongly prime. We do not know if this is the case for general superrings.
     
     \item If $A$ is a multiring, then every maximal ideal is prime (Proposition 1.7 of \cite{ribeiroanel}). For a general superring $A$, we do not know if a maximal ideal is prime.
     
     \item In his Ph.D Thesis \cite{ribeiroanel}, H. Ribeiro deals with elements \emph{weakly invertible} on a multiring $A$. This could be an anternative in dealing with the above questions.
 \end{enumerate}
\end{rem}

With all conventions and notations above, we obtain the following Lemma, which recover for superrings some properties holding for rings (and multirings).

\begin{lem}\label{lem1}
 Let $A$ be a superring and $I$ an ideal.
 \begin{enumerate}[i -]
  \item $I=A$ if and only if $1\in I$.
  \item $A/I$ is a superring. Moreover, if $A$ is full then $A/I$ is also full.
  \item $I$ is strongly prime if and only if $A/I$ is a superdomain.
 \end{enumerate}
 If $A$ is full, then
 \begin{enumerate}
     \item [iv -] $I=A$ if and only if $1\in I$, which occurs if and only if $A^*\cap I\ne\emptyset$ (in other words, if and only if 
$I$ contains an invertible element).
     \item [v -] $A$ is a superfield if and only if $\mathfrak I(A)=\{0,A\}$.
     \item [vi -] $I$ is maximal if and only if $A/I$ is a superfield.
 \end{enumerate}
\end{lem}

\begin{prop}
 Let $A$ be a superring and $I$ an ideal.
 \begin{enumerate}[i -]
     \item If $I$ is a maximal ideal, then it is prime.
     
     \item The ideal $I$ is prime if, and only if, $A/I$ is  quasi-superdomain\footnote{A superring $B$ is called quasi-superdomain if given $a,b \in B$ with $ab = \{0\}$, then $a = 0$ or $b = 0$}.
     
     \item (Prime Ideal Theorem) Let $S \subseteq A$ be a multiplicative set ($1 \in S$ and $S \cdot S \subseteq S$). Suppose that $S \cap I = \emptyset$. Then there is a prime ideal $p$ such that $I \subseteq p$ and $S \cap p = \emptyset$.
 \end{enumerate}
\end{prop}

\subsection{On commutative superrings of polynomials}

Even if the rings-like multi-algebraic structure have been studied for more than 70 years, the developments of notions of polynomials in the ring-like multialgebraic structure seems to have a more significant development only from the last decade: for instance in \cite{jun2015algebraic} some notion of multi polynomials is introduced to obtain some applications to algebraic and tropical geometry, in \cite{ameri2019superring} a more detailed account of variants of concept of multipolynomials over hyperrings is applied to get a form of Hilbert's Basissatz.

Here we will stay close to the perspective in  \cite{ameri2019superring}: let $(R,+,-,\cdot,0,1)$ be a superring and set
$$R[X]:=\{(a_n)_{n\in\omega}\in R^\omega:\exists\,t\,\forall n(n\ge t\rightarrow a_n=0)\}.$$
Of course, we define the \textbf{degree} of $(a_n)_{n\in\omega}\ne\bm0$ to be the smallest $t$ such that $a_n=0$ for all $n>t$. 

Now define the binary multioperations $+,\cdot :  R[X]\times R[X] \to  \mathcal P^*(R[X])$, a unary operation $-:R[X]\rightarrow R[X]$ and elements $0,1\in R[X]$ by
\begin{align*}
 (c_n)_{n\in\omega}\in (a_n)_{n\in\omega}+(b_n)_{n\in\omega}&\mbox{ iff }\forall\,n(c_n\in a_n+b_n) \\
  (c_n)_{n\in\omega}\in((a_n)_{n\in\omega}\cdot (b_n)_{n\in\omega}&\mbox{ iff }\forall\,n
  (c_n\in a_0\cdot b_n+a_1\cdot b_{n-1}+...+a_n\cdot b_0) \\
  -(a_n)_{n\in\omega}&=(-a_n)_{n\in\omega} \\
  0&:=(0)_{n\in\omega} \\
  1&:=(1,0,...,0,...)
\end{align*}
For convenience, we denote elements of $R[X]$ by $\alpha=(a_n)_{n\in\omega}$. Beside this, we denote
\begin{align*}
 1&:=(1,0,0,...), \\
 X&:=(0,1,0,...), \\
 X^2&:=(0,0,1,0,...)
\end{align*}
etc. In this sense, our ``monomial'' $a_iX^i$ is denoted by $(0,...0,a_i,0,...)$, where $a_i$ is in the $i$-th position; in 
particular, we will denote ${\underline{b}} = (b,0,0,...)$ and we frequently identify $b \in R \leftrightsquigarrow 
{\underline{b}} \in R[X]$.

The properties stated  in the Lemma below immediately follows from the definitions involving $R[X]$:

\begin{lem}\label{lemperm}
 Let $R$ be a superring and $R[X]$ as above and $n,m\in \mathbb N$.
 \begin{enumerate}[a -]
  \item $\{X^{n+m}\}=X^n\cdot X^m$.
  \item For all $a\in R$, $\{aX^n\}= {\underline{a}}\cdot X^n$.
  \item Given $\alpha=(a_0,a_1,...,a_n,0,0,...)\in R[X]$, with with $\deg\alpha \leq n$ and $m\ge1$, we have
  $$\alpha X^m=(0,0,...,0,a_0,a_1,...,a_n,0,0,...)=a_0X^m+a_1X^{m+1}+...+a_nX^{m+n}.$$
  \item For $\alpha=(a_n)_{n\in\omega}\in R[X]$, with $\deg\alpha=t$, 
  $$\{\alpha\}=a_0\cdot1+a_1\cdot X+...+a_t\cdot X^t=a_0+X(a_1+a_2X+...+a_nX^{t-1}).$$
  \item $R[X]$ is a superdomain iff $R$ is a superdomain.
  \item $R[X]$ is a superring.
  \item The map $a \in R \mapsto {\underline{a}} = (a,0, \cdots,0, \cdots)$ defines a full embedding $R\rightarrowtail R[X]$.
  \item For an ordinary ring $R$ (identified with a strict suppering), the superring $R[X]$ is naturally isomorphic to (the superring associated to) the ordinary ring of polynomials in one variable over $R$.
   \end{enumerate}
\end{lem}

Lemma \ref{lemperm} allow us to deal with the superring $R[X]$ as usual. In other words, we can assume that for $\alpha\in R[x]$, there exists 
$a_0,a_1,...,a_n\in R$ such that $\alpha=a_0+a_1X+...+a_nX^n$, and then, we can work simply denoting $\alpha=f(X)$, as usual. For example, combining the definitions and all facts above we get
$$(x-a)(x-b)=x^2+(a-b)x+ab=\{x^2+dx+e:d\in a-b\mbox{ and }e\in ab\}.$$

\begin{rem} 
If $R$ is a full superdomain, does not hold in general that $R[X]$ is also a full superdomain. In fact, even if $R$ is a 
hyperfield, there are examples, e.g. $R = K, Q_2$, such that $R[X]$ is not a full superdomain (see \cite{ameri2019superring}).
\end{rem}

\begin{defn}
 The superring $R[X]$ will be called the \textbf{superring of polynomials} with one variable over $R$. The elements of $R[X]$ will be called 
 polynomials. We denote $R[X_1,...,X_n]:=(R[X_1,...,X_{n-1}])[X_n]$.
\end{defn}

\begin{lem}[Adapted from Theorem 5 of \cite{ameri2019superring}]\label{degreelemma}
Let $R$ be a superring and $f,g\in R[X]\setminus\{0\}$.
\begin{enumerate}[i -]
    \item If $t(X)\in f(X)+g(X)$ and $f\ne-g$ then $$\min\{\deg(f),\deg(g)\}\le\deg(t)\le\max\{\deg(f),\deg(g)\}.$$
    \item If $R$ is a superdomain and $t(X)\in f(X)g(X)$, then $\deg(t)=\deg(f)+\deg(g)$. In particular, if $f_1(X),f_2(X),...,f_n(X)\ne0$ and $t(X)\in f_1(X)f_2(X)...f_n(X)$, then
    $$\deg(t)=\deg(f_1)+\deg(f_2)+...+\deg(f_n).$$
    \item (Partial Factorization) Let $R$ be a superdomain, $\deg(f)=n$ and $f\in (X-a_1)(X-a_n)...(X-a_p)$. Then $p=n$.
\end{enumerate}
\end{lem}

Let $f(X)=a_0+...+a_nX^n$ and $g(X)=b_0+...+b_mX^m$ with $a_n,b_m\ne0$. We establish the following notation: for $k\in\mathbb N$ with $k\le\deg(f)$ we define $(f)_k:=a_k$ (the $k$-th coefficient of $f$).

Despite the fact that $R[X]$ is not full in general, we have a powerful Lemma to get around this situation.

\begin{lem}\label{lemfator}
 Let $R$ be a superring and $f\in R[X]$ with $f(X)=a_nX^n+...+a_1X+a_0$. Then:
 \begin{enumerate}[i -]
     \item For all $b,c\in R$, $(b+cX)f(X)=bf(X)+cXf(X)$.
     \item For all $b,c\in R$ and all $p,q\in\omega$ with $p<q$, $$(bX^p+cX^q)f(X)=bX^pf(X)+cX^pf(X).$$
     \item For all $b,c,d\in R$ and all $p,q,r\in\omega$ with $p<q<r$, $$(bX^p+cX^q+dX^r)f(X)=bX^pf(X)+cX^pf(X)+dX^rf(X).$$
     \item For all $b_0,....,b_m\in R$,
     \begin{align*}
       (b_0+b_1X+b_2X^2+...+b_mX^m)f(X)=\\
     b_0f(X)+(b_1X+b_2X^2+...+b_mX^m)f(X).  
     \end{align*}
     \item For all $b_0,....,b_m\in R$,
     \begin{align*}
       (b_0+b_1X+b_2X^2+...+b_{m-1}X^{m-1}+b_mX^m)f(X)=\\
     (b_0+b_1X+b_2X^2+...+b_{m-1}X^{m-1})f(X)+b_mX^mf(X).  
     \end{align*}
     \item For all $b_0,....,b_m\in R$,
     \begin{align*}
       (b_0+b_1X+...+b_jX^j+b_{j+1}X^{j+1}+...+b_mX^m)f(X)=\\
     (b_0+b_1X+...+b_jX^j)f(X)+(b_{j+1}X^{j+1}+...+b_mX^m)f(X).  
     \end{align*}
     In particular, if $d\in R$, $g(X)\in R[X]$ and $r>\deg(g(X))$, then
     $$(g(X)+dX^r)f(X)=g(X)f(X)+dX^rf(X).$$
 \end{enumerate}
\end{lem}

\begin{teo}[Euclid's Division  Algorithm (3.4 in \cite{davvaz2016codes})]\label{euclid}
 Let $K$ be a superfield. Given polynomials $f(X),g(X)\in K[X]$ with $g(X)\ne0$, there exists $q(X),r(X)\in 
K[X]$ such that $f(X)\in q(X)g(X)+r(X)$, with $\deg r(X)<\deg g(X)$ or $r(X)=0$.
\end{teo}

\begin{rem}
$ $
 \begin{enumerate}[i -]
  \item Note that the polynomials $q$ and $r$ of Theorem \ref{euclid} are not unique in general: if $f\in gq+r$, then 
$f\in g(q+1-1)+r$ and $f\in gq+(r+1-1)$, then, if $\{0\}\ne1-1$, we have many $q$'s and $r$'s.

 \item However, if $R$ is a ring, then Theorem \ref{euclid} provide the usual Euclid Algorithm, with the uniqueness of the quotient and remainder.
\end{enumerate}
\end{rem}

\begin{teo}[Adapted from Theorem 6 of \cite{ameri2019superring}]\label{teoPID}
 Let $F$ be a full superfield. Then $F[X]$ is a principal ideal superdomain.
\end{teo}

\subsection{On Evaluation and Roots}


Let $R, S$ be  superrings and $h : R \to S$ be a morphism. Then $h$ extends naturally to a morphism in the superrings multipolynomials $h^X : 
R[X] \to S[X]$:
$$(a_n)_{n \in \mathbb{N}} \in R[X] \ \mapsto \ (h(a_n))_{n \in \mathbb{N}} \in S[X]$$

Now let $s \in S$. We define the $h$-\textbf{evaluation} of $s$ at $f(X)\in R[X]$ with $f(X)=a_0+a_1X+...+a_nX^n$ by
$$f^h(s)=ev^h(s,f):=\{s'\in S : s'\in h(a_0)+ h(a_1).s+h(a_2).s^2+...+h(a_n).s^n\}.$$
We define the $h$-\textbf{evaluation} for a subset $I\subseteq S$ by
$$f^h(I)=\bigcup_{s\in I}f^h(s).$$

In particular if $S\supseteq R$ are superrings and $\alpha\in S$, we have the \textbf{evaluation} of $\alpha$ at  $f(X)\in R[X]$ by
$$f(\alpha,S)=ev(\alpha,f,S)=\{b\in S: b \in a_0+a_1\alpha+a_2\alpha^2+...+a_n\alpha^n\}\subseteq S.$$
Note that the evaluation \textbf{depends} on the choice of $S$. When $S=R$ we just denote $f(\alpha,R)$ by $f(\alpha)$. 

A \textbf{root} of $f$ in $S$ is an element $\alpha\in S$ such that $0\in ev(\alpha,f,S)$. In this case we say that $\alpha$ is \textbf{$S$-algebraic} over $R$. An \textbf{effective root} of $f$ in $S$ is an element $\alpha\in S$ such that $f\in(X-\alpha)\cdot g(X)$ for some $g(X)\in R[X]$. A superring $R$ is \textbf{algebraically closed} if every non constant polynomial in $R[X]$ has a root in $R$. 

Observe that, if $F$ is a field, the evaluation of $F[X]$ as a ring coincide with the usual evaluation, and, of course, root and effective roots are the same thing. Therefore, if $F$ is algebraically closed as hyperfield and superfield, then will be algebraically closed in the usual sense. 

\begin{rem}  \label{unexpected-rem}
The  expansion of the above field-theoretical concepts to the multialgebraic theory  of superfields (hyperfields, in particular)  brings new phenomena:

\begin{enumerate}[i-]

\item (Polynomials can have infinite roots):
Let $F$ be a infinite pre-special hyperfield  (\cite{ribeiro2016functorial}). Then $F$ has characteristic $0$, $a^2=1$ for all $a\ne0$ so the polynomial $f(X)=X^2-1$ has infinite roots (i.e, $0\in ev(f,\alpha)$ for all $\alpha\in\dot F$).

\item (Finite hyperfields can be algebraically closed). The hyperfield $K=\{0,1\}$ is algebraically closed. In fact, if $p(X)=a_0+a_1X+a_2X^2+...+a_nX^n\in K[X]$, with $a_n\ne0$, then $0\in p(0)$ (if $a_0=0$) or $p(1)=K$, since $1+1=\{0,1\}$.

\end{enumerate}
\end{rem}

We have good results concerning irreducibly (see for instance, Theorem \ref{lemquadext} below). These results are the key to the development of superfields extensions, which leads us to some kind of algebraic closure.

\begin{defn}[Irreducibility]
Let $R$ be a superfield and $f,d\in R[X]$. We say that $d$ divides $f$ if and only if $f\in\langle d\rangle$, and denote $d|f$. 
We say that $f$ is \textbf{irreducible} if $\deg f\ge1$ and $u|f$ for some $u\in R[X]$ (i.e, $f\in\langle u\rangle$), then 
$\langle f\rangle=\langle u\rangle$. 
\end{defn}

\begin{teo}\label{lemquadext2}
Let $F$ be a full superfield and $p(X)\in F[X]$ be an irreducible polynomial. Then $\langle p(X)\rangle$ is a maximal ideal.
\end{teo}

If $F$ is not full, we cannot prove that $\langle p(X)\rangle$ is a maximal ideal. But we still have that $F[X]/\langle p\rangle$ is a superfield.
\begin{teo}\label{lemquadext}
 Let $F$ be a superfield and $p\in F[X]$ be an irreducible polynomial. Then $F[X]/\langle p\rangle$ is a superfield. In particular, $\langle p\rangle$ is a strongly prime.
\end{teo}



\begin{defn}
Let $F$ be a superfield and $p(X)\in F[X]$ be an irreducible polynomial. We denote $F(p):=F(p(X))=F[X]/\langle p(X)\rangle$.
\end{defn}

\begin{lem}\label{lemfator2}
 Let $F$ be a superfield and $p(X)\in F[X]$ be an irreducible polynomial. Denote $\overline X=\lambda$ and let $f\in F(p)$ with $f=\overline{a_n}\lambda^n+...+\overline{a_1}\lambda+\overline a_0$. Then:
 \begin{enumerate}[i -]
     \item For all $b,c\in F$, $(\overline b+\overline c\lambda)f=\overline bf+\overline c\lambda f$.
    
     \item For all $b_0,....,b_m\in F$,
     \begin{align*}
       (\overline{b_0}+\overline{b_1}\lambda+...+\overline{b_j}\lambda^j+\overline{b_{j+1}}\lambda^{j+1}+...+\overline{b_m}\lambda^m)f=\\
     (\overline{b_0}+\overline{b_1}\lambda+...+\overline{b_j}\lambda^j)f+(\overline{b_{j+1}}\lambda^{j+1}+...+\overline{b_m}\lambda^m)f.  
     \end{align*}
     In particular, if $d\in F$, $g\in F(p)$ with $g=\overline{b_0}+\overline{b_1}\lambda+\overline{b_2}\lambda^2+...+\overline{b_m}\lambda^m$ and $r>m$, then
     $$(g+\overline d\lambda^r)f=gf+\overline d\lambda^rf.$$
 \end{enumerate}
\end{lem}




\section{Matrices and determinants over commutative superrings} 

\begin{defn}
Let $m,n$ be positive integers. A \textbf{$m\times n$ matrix} over a commutative superring $R$ is  a double sequence $A$ of elements of $F$, distributed in $m$ rows and $n$ columns. The set of $m\times n$ matrices is denoted by $M_{m\times n}(R)$. When $m=n$, we denote $n\times n$ matrices by $M_n(R)$.
\end{defn}

A matrix $A\in M_{m\times n}(R)$ is represented simply by $A=(a_{ij})$ (with $m$ and $n$ subscript if necessary) or by a table as below:
 $$A=\begin{pmatrix}a_{11} & a_{12} & \ldots & a_{1n} \\
 a_{21} & a_{22} & \ldots & a_{2n} \\
 \vdots & \vdots & \ddots & \vdots \\
 a_{m1} & a_{m2} & \ldots & a_{mn}
 \end{pmatrix}$$
 
 For $A,B\in M_{n\times m}(R)$ and $\lambda\in R$ with $A=(a_{ij})$ and $B=(b_{ij})$ we define $-A=(-a_{ij})$ and (multi) operations 
 $$A+B:=\{(d_{ij}):d_{ij}\in a_{ij}+b_{ij}\mbox{ for all }i,j\} \neq \emptyset$$
 and 
 $$\lambda A=\{(d_{ij}):d_{ij}\in \lambda a_{ij}\mbox{ for all }i,j\}.$$ 
 If $A\in M_{n\times m}(R)$ with $A=(a_{ij})$ and $B\in M_{m\times p}(R)$ with $B=(a_{ij})$, we define $A\times B=AB\subseteq M_{n\times p}(R)$ by
 $$AB=\{(d_{ij}):d_{ij}\in \sum^n_{k=1}a_{ik}b_{kj}=a_{i1}b_{1j}+a_{i2}b_{2j}+...+a_{in}b_{nk}\mbox{ for all }i,j\}  \neq \emptyset.$$
 
  We denote $0=(0_{ij})\in M_{m\times n}(R)$ and $1=(\delta_{ij})_{ij}\in M_{n}(R)$ the usual zero and identity matrices respectively.
 
We say that $A\in M_n(R)$ is \textbf{invertible} iff there exist $B\in M_n(R)$ with $1 \in AB$ and $1\in BA$.
 
 Of course, we adopt freely the usual simplified notation from commutative algebra. For example for $A=(a_{ij})$ and $B=(b_{ij})$ we simply write
 $$\begin{pmatrix}a_{11} & a_{12} & \ldots & a_{1n} \\
 a_{21} & a_{22} & \ldots & a_{2n} \\
 \vdots & \vdots & \ddots & \vdots \\
 a_{m1} & a_{m2} & \ldots & a_{mn}
 \end{pmatrix}+\begin{pmatrix}b_{11} & b_{12} & \ldots & b_{1n} \\
 b_{21} & b_{22} & \ldots & b_{2n} \\
 \vdots & \vdots & \ddots & \vdots \\
 b_{m1} & b_{m2} & \ldots & b_{mn}
 \end{pmatrix}=\begin{pmatrix}a_{11}+b_{11} & a_{12}+b_{12} & \ldots & a_{1n}+b_{1n} \\
 a_{21}+b_{21} & a_{22}+b_{22} & \ldots & a_{2n}+b_{2n} \\
 \vdots & \vdots & \ddots & \vdots \\
 a_{m1}+b_{m1} & a_{m2}+b_{m2} & \ldots & a_{mn}+b_{mn}
 \end{pmatrix}$$
 with the analogous simplifications for $\lambda A$ and $AB$.

\begin{ex}
    Consider $X_2=\{-2,-1,0,1,2\}$ as in Example \ref{kaleid} and matrices $A,B,C\in M_{2}(X_2)$ given by $$A=\begin{pmatrix}1&1\\0&1\end{pmatrix},\,B=\begin{pmatrix}-1&1\\0&-1\end{pmatrix}\mbox{ and }C=\begin{pmatrix}2&0\\-1&2\end{pmatrix}.$$
    With our notations we have
    \begin{align*}
        &A+B=\begin{pmatrix}1-1&1+1\\0+0&1-1\end{pmatrix}=\\        &\left\lbrace\begin{pmatrix}-1&1\\0&-1\end{pmatrix},\begin{pmatrix}-1&1\\0&0\end{pmatrix},\begin{pmatrix}-1&1\\0&1\end{pmatrix},        \begin{pmatrix}0&1\\0&-1\end{pmatrix},\begin{pmatrix}0&1\\0&0\end{pmatrix},\begin{pmatrix}0&1\\0&1\end{pmatrix},\begin{pmatrix}1&1\\0&-1\end{pmatrix},\begin{pmatrix}1&1\\0&0\end{pmatrix},\begin{pmatrix}1&1\\0&1\end{pmatrix}\right\rbrace\\
        \qquad\\
        &A\cdot B=\begin{pmatrix}1\cdot(-1)+1\cdot0&1\cdot1+1\cdot(-1)\\0\cdot(-1)+1\cdot(0)&0\cdot1+1\cdot(-1)\end{pmatrix}=\begin{pmatrix}-1&1-1\\0&-1\end{pmatrix}=\\
        &\left\lbrace\begin{pmatrix}-1&-1\\0&-1\end{pmatrix},\begin{pmatrix}-1&0\\0&-1\end{pmatrix},\begin{pmatrix}-1&1\\0&-1\end{pmatrix}\right\rbrace\\
        \qquad\\
        &A\cdot C=\begin{pmatrix}1\cdot2+1\cdot(-1)&1\cdot0+1\cdot2\\0\cdot2+1\cdot(-1)&0\cdot0+1\cdot2\end{pmatrix}=\begin{pmatrix}2-1&2\\-1&2\end{pmatrix}=\left\lbrace\begin{pmatrix}2&2\\-1&2\end{pmatrix}\right\rbrace.
    \end{align*}
\end{ex}
 
 Here we will ''export'' the usual terminology of diagonal, triangular, block and elementary matrices available for fields to superfields.
 
 With these, using the fact that $(R,+,-,0)$ is a commutative multigroup we immediately have the following Theorem.
 \begin{teo}\label{matrix1}
 Let $R$ be a superring. Then $(M_{m\times n}(R),+,-,0)$ is a commutative multigroup.
 \end{teo}

For general a superring $R$, the matrix product in $M_n(R)$ is not associative in general (and of course, $M_n(R)$ is not a superring in general).

\begin{ex}
Let $R=X_2$ as in Example \ref{kaleid}. Of course, $R$ is not full because, for example, 
$$2(1-1)=\{-2,0,2\}\mbox{ and }2-2=R.$$
Now, consider the matrices
$$A=\begin{pmatrix}1&1\\0&1\end{pmatrix},\,B=\begin{pmatrix}-1&1\\0&-1\end{pmatrix}\mbox{ and }C=\begin{pmatrix}2&0\\-1&2\end{pmatrix}.$$
In fact we have
\begin{align*}
    (AB)C&=\left[\begin{pmatrix}1&1\\0&1\end{pmatrix}\begin{pmatrix}-1&1\\0&-1\end{pmatrix}\right]\begin{pmatrix}2&0\\-1&2\end{pmatrix}\\
    &=\begin{pmatrix}-1&1-1\\0&-1\end{pmatrix}\begin{pmatrix}2&0\\-1&2\end{pmatrix}=\begin{pmatrix}-2 - (1-1)&2(1-1)\\1&-2\end{pmatrix}
\end{align*}
and
\begin{align*}
    A(BC)&=\begin{pmatrix}1&1\\0&1\end{pmatrix}\left[\begin{pmatrix}-1&1\\0&-1\end{pmatrix}\begin{pmatrix}2&0\\-1&2\end{pmatrix}\right]\\
    &=\begin{pmatrix}1&1\\0&1\end{pmatrix}\begin{pmatrix}-2-1&2\\1&-2\end{pmatrix}=\begin{pmatrix}(-2-1)+1&2-2\\1&-2\end{pmatrix}.
\end{align*}
Then we have $(AB)C\ne A(BC)$.
\end{ex}

Despite the fact that $M_n(R)$ is not a superring in general, it is a structure of interest (as we will see, for example, in Theorem \ref{teohell2}) and in the following Lemma we collect the properties holding for the product in the general case. We assume that if $A$ and $B$ are matrices over $R$ and it is written $AB$, then the number of columns in $A$ and the number of rows in $B$ are the same (it is used similar convention for $A + B$).

\begin{lem}\label{matrix2}
 Let $R$ be a superring and $A,B,C$ matrices. Then:
 \begin{enumerate}[a -]
     \item $A\cdot0=0\cdot A=\{0\}$.
     \item If $m = n$, then $A\cdot1=1\cdot A=\{A\}$.
     \item If $A(B+C)\subseteq AB+AC$, with equality if $R$ is full.
     \item $(B+C)A\subseteq BA+CA$, with equality if $R$ is full.
     \item If $R$ is full, then $(AB)C=A(BC)$.
 \end{enumerate}
\end{lem}

Firstly, let us explicit the notation for $AB$: we write
$$AB=\begin{pmatrix}D_{11} & D_{12} & \ldots & D_{1n} \\
 D_{21} & D_{22} & \ldots & D_{2n} \\
 \vdots & \vdots & \ddots & \vdots \\
 D_{m1} & D_{m2} & \ldots & D_{mn}
 \end{pmatrix}$$
 or $AB=(D_{ij})$ where for all $i,j$, $D_{ij}$ is the set
 $$D_{ij}=\sum^n_{k=1}a_{ik}b_{kj}=a_{i1}b_{1j}+a_{i2}b_{2j}+...+a_{in}b_{nk}\mbox{ (of course, this is an equality of sets)}.$$
 Alternatively, we can proceed more directly, simply writing
 $$AB=\begin{pmatrix}\sum^n_{k=1}a_{1k}b_{k1} & \sum^n_{k=1}a_{1k}b_{k2} & \ldots & \sum^n_{k=1}a_{1k}b_{kn} \\
 \sum^n_{k=1}a_{2k}b_{k1} & \sum^n_{k=1}a_{2k}b_{k2} & \ldots & \sum^n_{k=1}a_{2k}b_{kn} \\
 \vdots & \vdots & \ddots & \vdots \\
 \sum^n_{k=1}a_{mk}b_{k1} & \sum^n_{k=1}a_{mk}b_{k2} & \ldots & \sum^n_{k=1}a_{mk}b_{kn}
 \end{pmatrix}$$
 Now we proceed with the proof of the Lemma.

\begin{proof}[Proof of Lemma \ref{matrix2}]
The argument here is in fact very similar to those one used in linear algebra over fields.
 \begin{enumerate}[a -]
     \item Let $A\cdot 0=(D_{ij})$ (as explained above). For all $i,j$ we have
     $$D_{ij}=\sum_{k}a_{ik}0_{kj}=\sum_{k}a_{ik}\cdot0=\sum_{k}0=0.$$
     Then $A\cdot0=\{0\}$. The same reasoning provide $0\cdot A=\{0\}$.
     
     \item Let $A\cdot1=(D_{i,j}).$ Since $1 = (\delta_{i,j})$ with $\delta_{i,j} \in \{0,1\}$ and $\delta_{i,j} = 1$ iff $i=j$, we have
     $$D_{i,j} = \sum_{k} a_{i,k}\delta_{k,j} = a_{i,j} \cdot 1 = \{a_{i,j}\}.$$
     Thus, $A \cdot 1 = \{A\}$. Similarly, $1 \cdot A = \{A\}$.
     \item 
     \begin{equation*}
         \begin{split}
             A(B + C) &= (a_{i,j})_{i,j} \cdot (b_{i,j} + c_{i,j})_{i,j}\\
                    & = \left(\sum_{k} a_{i,k}(b_{k,j}+ c_{k,j})\right)_{i,j}\\
                    & \subseteq \left(\sum_{k} a_{i,k}b_{k,j}+ a_{i,k}c_{k,j}\right)_{i,j}\\
                    & = \left(\sum_{k} a_{i,k}b_{k,j}\right)_{i,j}+\left(\sum_{k} a_{i,k}c_{k,j}\right)_{i,j} = AB + BC.
         \end{split}
     \end{equation*}
     
     When $R$ is full, it is immediate from above that $A(B+C) = AB + AC$.
     
     \item Similar argument as above.
     \item Assume that $R$ is full. Then
     \begin{equation*}
         \begin{split}
             (AB)C & = \left[ \left(a_{i,j}\right)_{i,j} \cdot (b_{i,j})_{i,j}\right] \cdot (c_{i,j})_{i,j} \\
                    & = \left(\sum_{k} a_{i,k}b_{k,j} \right)_{i,j} \cdot (c_{i,j})_{i,j}\\
                    & = \left( \sum_{l} \left(\sum_{k} a_{i,k}b_{k,l}\right) \cdot c_{l,j}\right)_{i,j}\\
                    & = \left( \sum_{l} \sum_{k} a_{i,k}b_{k,l} c_{l,j}\right)_{i,j}\\
                    & = \left( \sum_{k} a_{i,k} \cdot \left(\sum_{l} b_{k,l} c_{l,j}\right)\right)_{i,j}\\
                    & = \left(a_{i,j} \right) \cdot \left(\sum_{l} b_{i,l} c_{l,j}\right)_{i,j} = A(BC).
         \end{split}
     \end{equation*}
 \end{enumerate}
\end{proof}

In fact, with Theorem \ref{matrix1} and Lemma \ref{matrix2} we conclude the following.
\begin{teo}\label{matrix3}
For a superring $R$, if $R$ is full then $M_n(R)$ is a full superring, that is non-commutative if $n \geq 2$.
\end{teo}

We also have a generalized version of Lemma \ref{matrix2} (with the proof similar to the one given there).
\begin{lem}\label{matrix4}
 Let $R$ be a superring and $A,B,C,D,E,F$ matrices with $A\in M_{m\times n}(R)$, $B,C\in M_{n\times p}(R)$, $D,E\in M_{p\times m}(R)$ and $F\in M_{p\times q}(R)$. Then:
 \begin{enumerate}[a -]
     \item $A\cdot0_{n\times p}=\{0_{m\times p}\}$ and $0_{p\times n}\cdot A=\{0_{p\times n}\}$.
     \item $A\cdot1_{n\times n}=1_{m\times m}\cdot A=\{A\}$.
     \item $A(B+C)\subseteq AB+AC$, with equality if $R$ is full.
     \item $(D+E)A\subseteq DA+EA$, with equality if $R$ is full.
     \item If $R$ is full, then $(AB)F=A(BF)$.
 \end{enumerate}
\end{lem}

Despite the fact that full associativity do not hold in $M_n(R)$ for a general superring $R$, we have the following useful results. We start with a technical Definition:

\begin{defn}
    Let $R$ be a superring. We say that $R$ is \textbf{proto-full} if for all $a,b,c,d\in R$
    $$[(ab+ac)d]\cap[a(bd+cd)]\ne\emptyset.$$
\end{defn}

\begin{lem}
    Let $R$ be a proto-full superring. Then for all $a,b_1,...,b_n,d\in R$ we have
    $$[(ab_1+...+ab_n)d]\cap[a(b_1d+...+b_nd)]\ne\emptyset.$$
\end{lem}

Then rewriting the proof of Lemma \ref{matrix4}(e) we get the following.

\begin{lem}\label{matrix5}
    Let $R$ be a proto full superring and $A,B,C$ matrices with $A\in M_{m\times n}(R)$, $B\in M_{n\times p}(R)$ and $C\in M_{p\times q}(R)$. Then
    $$[(AB)C]\cap[A(BC)]\ne\emptyset.$$
\end{lem}

\begin{defn}\label{determinant}
Let $\mathcal A\subseteq M_n(R)$. We define the \textbf{determinant} of $\mathcal A$ as the \emph{subset} $\det(\mathcal A)\subseteq R$ given by the rule
$$\det(\mathcal A)=\bigcup_{A\in\mathcal A}\left\lbrace\sum_{\sigma\in S_n}\mbox{sgn}(\sigma)\prod^n_{j=1}a_{j\sigma(j)}\right\rbrace.$$
If $\mathcal A=\{A\}$ we simply write $\det(A)$ for the above formula.
\end{defn}

\begin{lem}[Properties of Determinant]
$A,B\in M_n(R)$, $A=(a_{ij})$, $B=(b_{ij})$ and $\lambda\in R$. Then:
\begin{enumerate}[a -]
    \item $\det(\lambda A)\subseteq\lambda^n\det(A)$, with equality if $R$ is full;
    \item if there is an entire row or column of zeros in $A$ then $\det(A)=\{0\}$.
    \item if $A=(a_{ij})$ is a triangular matrix (and in particular, diagonal matrix) then $\det(A)=a_{11}a_{22}...a_{nn}$.
\end{enumerate}
\end{lem}
\begin{proof}
$ $
\begin{enumerate}[a -]
    \item Using the very Definition we get
    \begin{align*}
        \det(\lambda A)&=\sum_{\sigma\in S_n}\mbox{sgn}(\sigma)\prod^n_{j=1}(\lambda a_{j\sigma(j)})
        =\sum_{\sigma\in S_n}\mbox{sgn}(\sigma)\lambda^n\prod^n_{j=1}a_{j\sigma(j)}\\
        &\subseteq \lambda^n\left(\sum_{\sigma\in S_n}\mbox{sgn}(\sigma)\prod^n_{j=1}a_{j\sigma(j)}\right)
        =\lambda^n\det(A)
    \end{align*}
    \item In this case we have $0\in\{a_{1\sigma(1)},a_{2\sigma(2)},...,a_{n\sigma(n)}\}$ for all $n\in S_n$. Then
    $$\prod^n_{j=1}a_{j\sigma(j)}=\{0\}\mbox{ for all }\sigma\in S_n,$$
    implying $\det(A)=\{0\}$.
    \item Follow immediately from Definition \ref{determinant}.
\end{enumerate}
\end{proof}

\section{Linear systems over superfields}

Throughout this entire Section fix a superfield $F$.
\begin{defn}\label{lin-equation}
A \textbf{linear equation} is an equation (term in the language of superfields) of type
$$Ax\subseteq B$$
where $A\in M_{1\times n}(F)$ ($n\in\mathbb N$), $x$ is a $n\times1$ vector of variables and $B\subseteq F$.
\end{defn}

\begin{rem}
We are defining ''linear equations'' (and more lately, ''linear systems'') in terms of matrices. We do this because for superfields we cannot agglutinate scalars and variables as we do in general linear algebra. For example, there is no reason for ''$a_1x_1+b_1x_1$'' be equal to ''$(a_1+b_1)x_1$''\footnote{Or saying in another words, in order to obtain $a_1x_1+b_1x_1=(a_1+b_1)x_1$ we should Define what would be a ''full'' variable, which is not a standard procedure in logic.}.    
\end{rem}

Despite this Remark, given a linear equation $Ax\subseteq B$, we can ''colloquially'' write
$$a_1x_1+a_2x_2+...+a_nx_n\subseteq B,$$
with $a_1,...,a_n\in F$. Of course, we could consider the equation
$$a_1x_1+a_2x_2+...+a_nx_n\supseteq B$$
as a linear one and proceed with two types of linear equations. But the type considered in \ref{lin-equation} seems to be (at first sight) more ''natural''.

We can use to this ''coloquial'' to write our equations (and further, systems), and while we are dealing with one equation (system), we will proceed with this ''coloquial'' language. But in order to get more general proofs and Definitions, we will always proceed with matrices.

\begin{defn}
 A \textbf{solution (weak solution)} of a linear equation $Ax\subseteq B$, is a matrix $d\in M_{n\times1}(F)$ such that $Ad\subseteq B$ ($Ad\cap B\ne\emptyset$). 
\end{defn}

\begin{defn}
A \textbf{linear system} is a conjunction of equations (term in the language of superfields) of type
$$Ax\subseteq B$$
where $A\in M_{m\times n}(F)$ ($n\in\mathbb N$), $x$ is a $n\times1$ vector of variables and $B\subseteq M_{m\times1}(F)$.
\end{defn}

 In this sense, a Linear system can be colloquially represented as usual
 $$S:\begin{cases}
 a_{11}x_1+a_{12}x_2+...+a_{1n}x_n\subseteq B_1\\
 a_{21}x_1+a_{22}x_2+...+a_{2n}x_n\subseteq B_2\\
 \vdots \\
 a_{m1}x_1+a_{m2}x_2+...+a_{mn}x_n\subseteq B_m\\
 \end{cases}$$
 
 A \textbf{(weak) solution} of a Linear system is a tuple $(d_1,...,d_n)$ such that $Ad\subseteq B$ ($Ad\cap B\ne\emptyset$).

\begin{defn}
A Linear system $Ax\subseteq B$ is \textbf{scaled} if $A$ is a upper triangular matrix.
\end{defn}

In the usual representation, a scaled linear system has the form:
 $$\begin{cases}
 a_{1r_1}x_1+\phantom{..........................}...\phantom{...}+a_{1n}x_n\subseteq B_1 \\
 \phantom{a_{1r_1}x_1+}a_{2r_2}x_2+\phantom{...............}...\phantom{...}+a_{1n}x_n\subseteq B_2 \\
 \vdots\\
 \phantom{a_{1r_1}x_1+a_{2r_2}x_2+}a_{nr_k}x_k+\phantom{..}...\phantom{...}+a_{nn}x_n\subseteq B_k
 \end{cases}$$
with $r_j\ge1$, and $a_{jr_j}\ne0$, $j=1,...,k$ e $r_1<r_2<...<r_k$. For a scaled system we have three situations:
\begin{enumerate}[I -]
     \item The last equation is of type
     $$0x_1+...+0x_n\subseteq B_p\mbox{ with }0\notin B_p.$$
     In this case $S$ is impossible.
     
     \item There is no equation of type (I) and $p=n$. 
     
     \item There is no equation of type (I) and $p<n$. 
 \end{enumerate}
 Suppose $S$ of type (II). Then we have a situation
  $$\begin{cases}
 a_{11}x_1+\phantom{..........................}...\phantom{...}+a_{1n}x_n\subseteq B_1 \\
 \phantom{a_{22}x_2+}a_{23}x_3+\phantom{...............}...\phantom{...}+a_{2n}x_n\subseteq B_2 \\
 \vdots\\
 \phantom{a_{1r_1}x_1+a_{2r_2}x_2+a_{nr_k}x_k+}...+a_{nn}x_n\subseteq B_n
 \end{cases}$$
 with $a_{ii}\ne0$ for all $i=1,...,n$. Getting $x_1,...,x_n$ recursively by suitable choices
 \begin{align}\label{rec-sol}
 \begin{cases}
     x_n&\in a^{-1}_{nn}B_n \\
     x_{k}&\in a^{-1}_{kk}B_k-a^{-1}_{kk}a_{k(k+1)}x_{k+1}-...-a^{-1}_{kk}a_{kn}x_{n}\mbox{ for }k=n-1,...,1
 \end{cases}
 \end{align}
 we have a weak solution of the system $S$ (this solution is weak basically because $a^{-1}_{kk}a_{kn}$ is not a singleton in general). The same reasoning shows that for a scaled system of type (III), we can find a parametric weak solution for the system. 
 
 \begin{ex}\label{sistem-solved}
 Consider $n=2$ and the system over $F$,
 $$\begin{pmatrix}a & b\\ 0 & c\end{pmatrix}\begin{pmatrix}x\\ y\end{pmatrix}\subseteq\begin{pmatrix}D_1\\ D_2\end{pmatrix}$$
 or in our ''coloquial'' representation, the system
 $$\begin{cases}
     ax+by\subseteq D_1 \\\phantom{1ax+}cy\subseteq D_2
 \end{cases}\mbox{ with }a,c\ne0.$$
 Since $D_2\subseteq c(c^{-1}D_2)$, for all $d_2\in D_2$ there exist $z\in c^{-1}D_2$ with $d_2\in cz$. Pick $y_0=z$. So we have $$cy_0\subseteq c(c^{-1}D_2)\mbox{ with }cy_0\cap D_2\ne\emptyset.$$ Hence we get $y_0\in c^{-1}D_2$ and we can choose $x_0\in a^{-1}D_1-by_0$ in order to obtain
$$ax_0+by_0\subseteq a(a^{-1}D_1-by_0)+by_0\subseteq aa^{-1}D_1-aa^{-1}by_0+by_0\mbox{ and }ax_0+by_0\cap D_1\ne\emptyset.$$  
 \end{ex}

 Of course, linear systems over superfields yields to more flexibility than linear systems over fields. It is ''easier'' to get a weak solution of a linear systems over superfields than over a field as we see in the Example below.

\begin{ex}
Let $F=\mathbb Q/_m\mathbb Q^{*2}$. We have $2,5\notin D(\langle1,1\rangle)$, because $2=1\cdot1^2+1\cdot1^2$ and $5=1\cdot2^2+1\cdot1^2$. Then $\overline2,\overline5\in\overline1+\overline1$ and the system
$$\begin{cases}
     x+y\subseteq\{\overline1\} \\\phantom{ax+}y\subseteq \{\overline5\}
 \end{cases}$$
 over $F$ has at least a weak solution $x=y=1$.
\end{ex}
 
\begin{defn}[Elementary Operations]
Let $A\in M_{m\times n}(F)$. The \textbf{elementary operations are}:
\begin{enumerate}[I -]
    \item \textbf{Permute} lines $i$ e $j$; which will be indicated by $L_i\leftrightarrow L_j$;
    \item \textbf{Multiply} each coefficient of a line $i$ by an element $\lambda\ne0$ in $F$; which will be indicated by $L_i\leftarrow\lambda L_i$;
    \item \textbf{Sum} line $i$ with line $j$ and keep the result in line $i$; which will be indicated by $L_i\leftarrow L_i+L_j$.
\end{enumerate}
\end{defn} 

Of course, given a linear system $Ax\subseteq B$, we generate more than one system after the application of a sequence of elementary operations on the matrices $A$ and $B$. We denote the systems obtained by a set of systems $Ax\subseteq B$ (with $A\subseteq M_{m\times n}(F)$, $B\subseteq M_{m\times 1}(F)$) after the sequence of elementary operations $O=\{o_1,...,o_n\}$ by $(Ax\subseteq B)^O$.

The elementary operations defined above could be described in terms of matrix multiplication (as we usually do for fields). For example, considering the matrix $A\in M_{2\times 2}(F)$ given by
$$A=\begin{pmatrix}a & b\\ c & d\end{pmatrix}$$
the application of $L_1\leftrightarrow L_2$ is just
$$\begin{pmatrix}0 & 1\\ 1 & 0\end{pmatrix}\begin{pmatrix}a & b\\ c & d\end{pmatrix}=\begin{pmatrix}c & d\\ a & b\end{pmatrix}$$

If $R$ is a proto-full superfield, to realize an elementary operation on the system $Ax\subseteq B$ is equivalent to multiply $A$ and $B$ by an elementary matrix\footnote{Elementary matrices are a standard topic in many Linear Algebra books, but For a quick reference, consult \url{https://en.wikipedia.org/wiki/Elementary_matrix}.} $E\in M_{m\times m}(F)$ in order to obtain the system $(EA)x\subseteq (EB)$.

\begin{lem}\label{element-sol}
Let $Ax\subseteq B$ be a set of Linear systems and $O=\{o_1,...,o_n\}$. Then every solution of a system in $Ax\subseteq B$ is a solution in some system in $(Ax\subseteq B)^O$. 
If $F$ is full, then every weak solution of a system in $Ax\subseteq B$ is a weak solution in some system in $(Ax\subseteq B)^O$.
\end{lem}
\begin{proof}
We only need to deal with the elementary operations. Consider a system
$$\begin{pmatrix}
 a_{11} & a_{12} & \ldots & a_{1n} \\
 \vdots & \vdots & \ddots & \vdots \\
 a_{m1} & d_{m2} & \ldots & a_{mn}
 \end{pmatrix}
 \begin{pmatrix}x_1\\x_2\\\vdots\\x_n\end{pmatrix}
 \subseteq
 \begin{pmatrix}B_1\\B_2\\\vdots\\B_n\end{pmatrix}$$
 We already know that operations (I) and (II) preserves solutions. Now consider without loss of generalization the elementary operation $L_1\leftarrow L_1+L_2$. Then we arrive at the set of systems
 $$\begin{pmatrix}
 a_{11}+a_{21} & a_{12}+a_{22} & \ldots & a_{1n}+a_{2n} \\
 \vdots & \vdots & \ddots & \vdots \\
 a_{m1} & d_{m2} & \ldots & a_{mn}
 \end{pmatrix}
 \begin{pmatrix}x_1\\x_2\\\vdots\\x_n\end{pmatrix}
 \subseteq
 \begin{pmatrix}B_1+B_2\\B_2\\\vdots\\B_n\end{pmatrix}$$
 If we get $d_1,....,d_n$ with
 $$\begin{pmatrix}
 a_{11} & a_{12} & \ldots & a_{1n} \\
 \vdots & \vdots & \ddots & \vdots \\
 a_{m1} & d_{m2} & \ldots & a_{mn}
 \end{pmatrix}
 \begin{pmatrix}d_1\\d_2\\\vdots\\d_n\end{pmatrix}
 \subseteq \begin{pmatrix}B_1\\B_2\\\vdots\\B_n\end{pmatrix}$$
 in particular
 $$(a_{11}d_1+...+d_{1n}d_n)+(a_{21}d_1+...+d_{2n}d_n)\subseteq B_1+B_2,$$
 and
 $$\begin{pmatrix}
 a_{11}+a_{21} & a_{12}+a_{22} & \ldots & a_{1n}+a_{2n} \\
 \vdots & \vdots & \ddots & \vdots \\
 a_{m1} & d_{m2} & \ldots & a_{mn}
 \end{pmatrix}
 \begin{pmatrix}d_1\\d_2\\\vdots\\d_n\end{pmatrix}
\subseteq\begin{pmatrix}B_1+B_2\\B_2\\\vdots\\B_n\end{pmatrix}$$
proving (after induction) that every solution of $Ax\subseteq B$ is a solution $(Ax\subseteq B)^O$. For the weak solution part,
 suppose $F$ is full and $d_1,....,d_n$ with
 $$\begin{pmatrix}
 a_{11} & a_{12} & \ldots & a_{1n} \\
 \vdots & \vdots & \ddots & \vdots \\
 a_{m1} & d_{m2} & \ldots & a_{mn}
 \end{pmatrix}
 \begin{pmatrix}d_1\\d_2\\\vdots\\d_n\end{pmatrix}
 \cap\begin{pmatrix}B_1\\B_2\\\vdots\\B_n\end{pmatrix}\ne\emptyset$$
 In particular
 $$[(a_{11}+a_{21})d_1+...+(a_{1n}+a_{2n})d_n]=(a_{11}d_1+...+d_{1n}d_n)+(a_{21}d_1+...+d_{2n}d_n)\cap (B_1+B_2)\ne\emptyset,$$
 
 and
 $$\begin{pmatrix}
 a_{11}+a_{21} & a_{12}+a_{22} & \ldots & a_{1n}+a_{2n} \\
 \vdots & \vdots & \ddots & \vdots \\
 a_{m1} & d_{m2} & \ldots & a_{mn}
 \end{pmatrix}
 \begin{pmatrix}d_1\\d_2\\\vdots\\d_n\end{pmatrix}
\cap\begin{pmatrix}B_1+B_2\\B_2\\\vdots\\B_n\end{pmatrix}\ne\emptyset.$$
proving (after induction) that every weak solution of $Ax\subseteq B$ is weak a solution $(Ax\subseteq B)^O$.
\end{proof}

Given a system $Ax\subseteq B$, we can obtain a set of scaled systems $(Ax\subseteq B)^{scaled}$,  after a finite sequence of elementary operations in the same way as usual. Unfortunately, despite the result obtained in Lemma \ref{element-sol} we do not know if solutions of $(Ax\subseteq B)^{scaled}$ are solutions of $Ax\subseteq B$. 


From now on, given a system $Ax\subseteq B$, \textbf{to solve} $Ax\subseteq B$ will have the meaning \textbf{to find a weak solution of} $Ax\subseteq B$, and a \textbf{$n\times n$ system} will mean a system $Ax\subseteq B$ with $A\in M_{n\times n}(F)$ (and $B\in M_{n\times 1}(F)$).

\begin{defn}
Let $A=(a_{ij})\in M_n(F)$ and denote $A_i=(a_{i1},...,a_{in})$ the $i$-th row and $A^j=(a_{1j},...,a_{nj})$ the $j$-th column We say that $A_i$ is a linear combination of $\{A_1,...,A_{i-1},A_{i+1},...,A_n\}$ if there exist $a_1,...,a_{i-1},a_{i+1},...,a_n\in F$ such that
$$A_i\cap\left[\left(\sum^{r_1}_{j=1}\lambda_{j1}\right)A_1+...+\left(\sum^{r_{i-1}}_{j=1}\lambda_{j(i-1)}\right)A_{i-1}+\left(\sum^{r_{i+1}}_{j=1}\lambda_{j(i+1)}\right)A_{i+1}+...+\left(\sum^{r_n}_{j=1}\lambda_{jn}\right)A_n\right]\ne\emptyset.$$
\end{defn}


    

\begin{lem}\label{scal4}
    Let $F$ be a superfield and $A\in M_n(F)$ a upper triangular matrix. Then $A$ is invertible iff $0\notin\det(A)$.
\end{lem}
\begin{proof}
    Let
    $$A=
    \begin{pmatrix}a_{11} & a_{12} & \ldots & a_{1n} \\
     0 & a_{22} & \ldots & a_{2n} \\
     \vdots & \vdots & \ddots & \vdots \\
     0 & 0 & \ldots & a_{nn}
     \end{pmatrix}
    $$
    We already know that $\det(A)=a_{11}a_{22}...a_{nn}$. Then, we need to prove that $A$ is invertible iff $a_{ii}\ne0$ for all $i=1,2,...,n$.
    
    Now, let $B\in M_n(F)$ be another upper triangular matrix, saying
    $$B=
    \begin{pmatrix}b_{11} & b_{12} & \ldots & b_{1n} \\
     0 & b_{22} & \ldots & b_{2n} \\
     \vdots & \vdots & \ddots & \vdots \\
     0 & 0 & \ldots & b_{nn}
     \end{pmatrix}$$
     We also know that
     $$AB=
     \begin{pmatrix}a_{11}b_{11} & a_{11}b_{12}+a_{12}b_{22} & a_{11}b_{13}+a_{12}b_{23}+a_{13}b_{33} & \ldots & \sum^n_{k=1}a_{1k}b_{kn} \\
     0 & a_{22}b_{22} & a_{22}b_{23}+a_{23}b_{33}& \ldots & \sum^n_{k=2}a_{2k}b_{kn} \\
     \vdots & \vdots & \vdots & \ddots & \vdots \\
     0 & 0 & 0 & \ldots & a_{nn}b_{nn}
     \end{pmatrix}$$
     and
     $$BA=
     \begin{pmatrix}b_{11}a_{11} & b_{11}a_{12}+b_{12}a_{22} & b_{11}a_{13}+b_{12}a_{23}+b_{13}a_{33} & \ldots & \sum^n_{k=1}b_{1k}a_{kn} \\
     0 & b_{22}a_{22} & b_{22}a_{23}+b_{23}a_{33}& \ldots & \sum^n_{k=2}b_{2k}a_{kn} \\
     \vdots & \vdots & \vdots & \ddots & \vdots \\
     0 & 0 & 0 & \ldots & b_{nn}a_{nn}
     \end{pmatrix}$$
     Then, if $a_{ii}=0$ for some $i\in\{1,2,...,n\}$, we have $I_n\notin AB\cap BA$ for all $B\in M_n(F)$. This is enough to prove that $A$ cannot be invertible.

     Now, suppose $a_{ii}\ne0$ for all $i=1,2,...,n$. We will choose the elements $b_{ij}$ in order to get 
     $$I_n\in AB\cap BA.$$
     First, choose $b_{ii}=a_{ii}^{-1}$. Then, considering $AB=(P_{ij})$ and $BA=(Q_{ij})$, we want to get $0\in P_{ij}$ and $0\in Q_{ij}$ for all $i\ne j$. We need to choose $b_{(n-1)n}$ in order to get
     \begin{align*}
         0&\in a_{(n-1)(n-1)}b_{(n-1)n}+a_{(n-1)n}b_{nn}\mbox{ and }\\
         0&\in b_{(n-1)(n-1)}a_{(n-1)n}+b_{(n-1)n}a_{nn}.
     \end{align*}
     Then (remember that $b_{ii}=a_{ii}^{-1}$) we need
     $$b_{(n-1)n}\in -a_{(n-1)n}a_{(n-1)(n-1)}^{-1}a_{nn}^{-1}.$$
     Then we choose $b_{nn},b_{(n-1)(n-1)}$ and $b_{(n-1)n}$ (i,e, we complete the process for the $n$-th and $(n-1)$-th rows of $B$).

     Now, we need to choose $b_{(n-2)(n-1)}$ and $b_{(n-2)n}$ in order to get
     \begin{align*}
         0&\in a_{(n-2)(n-2)}b_{(n-2)(n-1)}+a_{(n-2)(n-1)}b_{(n-1)(n-1)}+a_{(n-2)n}b_{n(n-1)}\mbox{ and }\\
         0&\in a_{(n-2)(n-2)}b_{(n-2)n}+a_{(n-2)(n-1)}b_{(n-1)n} +a_{(n-2)n}b_{nn}
     \end{align*}
     and
     \begin{align*}
         0&\in b_{(n-2)(n-2)}a_{(n-2)(n-1)}+b_{(n-2)(n-1)}a_{(n-1)(n-1)}+b_{(n-2)n}a_{n(n-1)}\mbox{ and }\\
         0&\in b_{(n-2)(n-2)}a_{(n-2)n}+b_{(n-2)(n-1)}a_{(n-1)n} +b_{(n-2)n}a_{nn}
     \end{align*}
     Picking $b_{(n-2)(n-1)}$ and $b_{(n-2)n}$ such that
     \begin{align*}
         b_{(n-2)n}&\in-a_{(n-2)(n-2)}^{-1}a_{(n-2)(n-1)}b_{(n-1)n} -a_{(n-2)(n-2)}^{-1}a_{(n-2)n}b_{nn}\mbox{ and }\\
         b_{(n-2)(n-1)}&\in-a_{(n-2)(n-2)}^{-1}a_{(n-2)(n-1)}b_{(n-1)(n-1)}-a_{(n-2)(n-2)}^{-1}a_{(n-2)n}b_{n(n-1)}
     \end{align*}
     we complete the process for the $n$-th, $(n-1)$-th and $(n-2)$-th rows of $B$. Repeating this process more $n-3$ times we arrive at a matrix $B$ such that $I_n\in AB\cap BA$, as desired.
\end{proof}



\section{Applications to the theory of algebraic extensions of superfields}

We have some possibilities to consider in order to define the notion of extension for superfields:

 \begin{defn}[Extensions]\label{extension}
Let $F$ and $K$ be superfields.
\begin{enumerate}[i -]
    \item We say that $K$ is a \textbf{proto superfield extension (or just a proto extension)} of $F$, notation $K|_pF$, if $F\subseteq K$.
    \item We say that $K$ is a \textbf{superfield extension (or just an extension)} of $F$, notation $K|F$ if $F\subseteq K$ and the inclusion map $F\hookrightarrow K$ is a superfield morphism.
    \item We say that $K$ is a \textbf{full superfield extension (or just a full extension)} of $F$, notation $K|_fF$ if $F\subseteq K$ and the inclusion map $F\hookrightarrow K$ is a full superfield morphism.
\end{enumerate}
\end{defn}

\begin{ex}
$ $
\begin{enumerate}[i -]
    \item Of course, all full extension is an extension and all extension is a proto extension.
    
    \item We have $K\subseteq Q_2$ but the inclusion map $K\hookrightarrow Q_2$ is not a morphism. Then we have a proto extension $Q_2|_pK$ that is not an extension.
    
    \item For $p,q$ prime integers with $q\ge p$ we have an inclusion morphism $H_p\hookrightarrow H_q$, but this morphism is not full. Then we have an extension $H_q|H_p$ that is not a full extension.
    
    \item Let $F$ be a superfield, $p\in F[X]$ an irreducible polynomial and $F(p)=F[X]/\langle p\rangle$. Then we have a full morphism $F\hookrightarrow F(p)$ so we have a full extension $F(p)|_fF$.
    
    \item Let $F,K$ be fields such that $F\subseteq K$. Then the field extension $K|F$ satisfy all conditions in Definition \ref{extension}.
\end{enumerate}
\end{ex}

The result below justify a deeper look at full superfield extensions.

\begin{teo}\label{unicityext}
Let $K_1|_fF$ and $K_2|_fF$ be full superfield extensions and suppose that $\gamma\in K_1\cap K_2$. Then
$$F[\gamma,K_1]=F[\gamma,K_2].$$
\end{teo}

\begin{defn}[Algebraic Extensions]
We say that a proto extension $K|_pF$ is \textbf{algebraic} if all element $\alpha\in K$ is $K$-algebraic over $F$. We denote the same for extensions and full extensions.
\end{defn}

\begin{defn}[Linear Independency, Basis, Degree]
Let $K|_pF$ be a proto extension and $I\subseteq K$. We say that $I$ is \textbf{$F$-linearly independent} if for all 
distinct $\lambda_1,...,\lambda_n\in I$, $n\in\mathbb N$, the following hold:
$$\mbox{If }0\in \left(\sum^{r_1}_{j=1}a_{j1}\right)\lambda_1+...+ \left(\sum^{r_1}_{j=1}a_{n1}\right)\lambda_n \mbox{ then }0\in\left(\sum^{r_1}_{j=1}a_{ij}\right)\mbox{ for all }j.$$
and $I$ is \textbf{$F$-linearly dependent} if it is not $F$-linearly independent. We say that $I$ is a \textbf{$F$-basis} of 
$K$ if $I$ is linearly independent and $K$ is \textbf{generated by $I$}, i.e,
$$K=\bigcup_{n\ge0}\left\lbrace\sum^n_{i=0}a_i\lambda_i:a_i\in F,\,\lambda_i\in I\right\rbrace.$$
In this case, we write $K=F[I]$. We define the \textbf{degree} of $K|_pF$, notation $[K:F]$, by 
the following
$$[K:F]:=\infty\mbox{ or }[K:F]:=\max\{n:\mbox{the set }\{1,\lambda,\lambda^2,...,\lambda^n\}\mbox{ is linearly independent for all }\lambda\in K\}.$$
\end{defn}

\begin{rem}\label{rem1}
There are these immediate consequences of the above definitions:
\begin{enumerate}[a -]
\item Here the Definitions of linear independence, generators and dimension are ad hoc (in the sense the until now, we do not know if there are some ''theory of linear algebra'' available for superfields).
 \item If $I\subseteq K$ is linearly independent and $J\subseteq I$ then $J$ is also linearly 
independent.
 \item An element $\alpha\in K$ is $F$-algebraic if and only if $\{\alpha^k:k\in\mathbb N\}$ is $F$-linearly dependent.
 \item If $[K:F]<\infty$ then all $\alpha\in K$ is $F$-algebraic.
 \item Let $F$ be a superfield and $p\in F[X]$ an irreducible polynomial, say $p(X)=a_0+a_1X+...+a_nX^{n-1}+X^n$. Then $\{\overline1,\overline X,...,\overline X^{n-1}\}$ is a $F$-basis of $F(p)$.
\end{enumerate}
\end{rem}

Now, let $K|_pF$ be a proto extension and $\gamma\in K$ algebraic. Then there exist an irreducible polynomial $f(X)$ such that $0\in f(\gamma,K)$. Let $\mbox{Irr}_F(\gamma,K)$ be the minimum degree irreducible polynomial $f(X)$ 
such that $0\in f(\gamma,K)$. Let $F[\gamma,K]\subseteq K$ be the set
$$F[\gamma,K]:=\bigcup_{f\in F[X]}ev(f,\gamma,K)\subseteq K,$$
and $I_{\gamma,K}\subseteq F[\gamma,K]$ the set
$$I_{\gamma,K}:=\bigcup_{f\in \langle\mbox{Irr}_F(\gamma,K)\rangle}ev(f,\gamma,K)\subseteq K.$$
Note that for all $g\in F[X]$ and all $a_0,...,a_n\in F$, applying the ``Newton's binom formula'' we get
$$ev(g,(a_0+a_1\gamma+a_2\gamma^2+...+a_{n-1}\gamma^{n-1}+a_n\gamma^n),K)\subseteq F[\gamma,K].$$

\begin{rem}
 $ $
 \begin{enumerate}[i -]
    \item If $K|F$ is a field extension then our $F[\gamma,K]$ coincide with the usual simple extension $F(\gamma)$.
    
    \item If $K|F$ is a superfield extension and $\gamma\in K$, then $F[\gamma,K]$ \textbf{depends on the choice of $K$}. For example, consider $H_3|H_1$ and $H_5|H_1$ and the element $2\in H_3$ (and of course, in $H_5$). Then
    \begin{align*}
      H_2[2,H_3]&=\bigcup_{f\in H_2[X]}ev(f,\gamma,H_3)=H_3, \\
      H_2[2,H_5]&=\bigcup_{f\in H_2[X]}ev(f,\gamma,H_5)=H_5,
    \end{align*}
    and then, $H_2[2,H_3]\ne H_2[2,H_5]$.
  
  \item For a proto extension $K|_pF$ the set $F[\gamma,K]$ may not be a superfield! Let $F=H_2$, $K=\mathbb R$ and $\gamma=2$. Then
  $$H_2[2,\mathbb R]=2\mathbb Z$$
  which is not a superfield.
 \end{enumerate}
\end{rem}

At this point, our goal is to obtain an appropriate notion for simple extensions of superfields. In other words, given a full extension $K|_fF$ and $\alpha\in K$ algebraic, it is highly desirable to obtain a superfield $F(\alpha)$ that:
\begin{enumerate}
    \item $F\cup\{\alpha\}\subseteq F(\alpha)$;
    \item $F(\alpha)$ is the minimal superfield (with respect to inclusion) satisfying (1);
    \item $F(\alpha)$ is "computable" in some way (or saying it in a more realistic manner, we want that $F(\alpha)\cong F(p)$ with $p(X)=\mbox{Irr}_F(\alpha)$)\footnote{As we will see later, simple calculations with superfield are highly demanding...}.
\end{enumerate}

For general superfields there are some obstacles to achieve this goal. The very first one is the fact that $R[X]$ is not full in general. However, we have an interesting property valid for all $a,b\in R[X]$:
$$a(1+X)=a+aX\mbox{ and }(a+b)X=aX+bX.$$
This property is the inspiration for the following definition.

\begin{defn}\label{almostfull}
Let $K|_pF$ be a proto superfield extension and $\gamma\in K$. Suppose that $K$ is $F$-generated by $\{1,\gamma^2,...,\gamma^n\}$. We say that $K$ is \textbf{$F$-almost full relative to $\gamma$ (or just almost full)} if for all $a,b,c\in F$, and all $p,q,r\in\mathbb N$ distinct
$$(a\gamma^p+b\gamma^q+c\gamma^r)\gamma=a\gamma^{p+1}+b\gamma^{q+1}+c\gamma^{r+1}.$$
\end{defn}

Here are some immediate consequences of Definition \ref{almostfull}:

\begin{lem}\label{lemfator3}
 Let $K|_fF$ be a full extension $F$-almost full relative to $\gamma$ and let 
 $$A=a_0+a_1\gamma+a_2^2+...+a_n\gamma^n.$$
 Then:
 \begin{enumerate}[i -]
     \item For all $b,c\in F$, $(b+c\gamma)A=bA+c\gamma A$.
     \item For all $b_0,....,b_m\in F$,
     \begin{align*}
      (b_0+b_1\gamma+...+b_j\gamma^j+b_{j+1}\gamma^{j+1}+...+b_m\gamma^m)A=\\
     (b_0+b_1\gamma+...+b_j\gamma^j)A+(b_{j+1}\gamma^{j+1}+...+b_m\gamma^m)A.
     \end{align*}
     In particular, if $d\in F$, $B\subseteq K$ with $B=b_0+b_1\gamma+b_2\gamma^2+...+b_m\gamma^m$ and $r>m$, then
     $$(B+d\gamma^r)A=AB+d\gamma^rA.$$
 \end{enumerate}
\end{lem}

\begin{lem}\label{almostfact}
 Let $K|_fF$ be a full extension $F$-almost full relative to $\gamma$. Then:
 \begin{enumerate}[i -]
     \item $K=F[\gamma,K]$;
     \item If $K|_fF$ and $L|_fK$ are almost full then $L|F$ is almost full;
     \item If $L|_fF$ is another full extension and $\pi:K\rightarrow L$ is a full surjective morphism, then $L|_fF$ is $F$-almost full relative to $\pi(\gamma)$;
     \item For all $a_0,...,a_n,b_0,...,b_n\in F$,
     \begin{align*}
         &(a_0+a_1\gamma+a_2\gamma^2+...+a_{n-1}\gamma^{n-1}+a_n\gamma^n)
         (b_0+b_1\gamma+b_2\gamma^2+...+b_{n-1}\gamma^{n-1}+b_n\gamma^n)\subseteq\\
         &a_0b_0+\left(\sum^1_{j=0}a_jb_{1-j}\right)\gamma+...+\left(\sum^{2n-1}_{j=0}a_jb_{(2n-1)-j}\right)\gamma^{2n-1}+\left(\sum^{2n}_{j=0}a_jb_{1-j}\right)\gamma^{2n}
     \end{align*}
     with the convention $a_j=b_j=0$ if $j>n$.
 \end{enumerate}
\end{lem}

Let $K|_fF$ be a full extension and $\alpha\in K$ algebraic over $F$. 
Our aim is to provide an almost full algebraic extension $F(\alpha)|_fF$ containing $F$ and $\alpha$. The key to that is to find a way to describe algebraic elements of $K$. Here we have a first result in this direction.

\begin{teo}[Almost Full Newton's Binom]
Let $K|F$ be an almost full superfield extension $F$-generated by $\{1,\gamma,...,\gamma^n\}$, $\gamma\in K$. Then for all $a,b\in F$,
$$(a+b\gamma)^n=\sum^n_{j=0}\binom{n}{j}a^j(b\gamma)^{n-j}.$$
\end{teo}
\begin{proof}
By induction is enough to prove the case $n=2$. We have
  \begin{align*}
    (a+b\gamma)^2&:=(a+b\gamma)(a+b\gamma)\stackrel{\ref{lemfator3}}{=} a(a+b\gamma)+b\gamma(a+b\gamma)= a^2+ab\gamma+b\gamma a+(b\gamma)^2 \\
    &=a^2+ab\gamma+ab\gamma+(b\gamma)^2=a^2+2ab\gamma+(b\gamma)^2:=\sum^2_{j=0}\binom{n}{j}a^j(b\gamma)^{n-j}.
  \end{align*}
\end{proof}

As an application of the theory of linear systems over superfields, we present in the sequel a key result on the theory of algebraic extensions of superfields, which states that our "best candidate for simple extension", $F(p)$, is an full algebraic and almost full extension of $F$. We start proving that $F(p)$ is proto-full.

\begin{teo}\label{teohell2}
Let $F$ be a superfield and $p\in F[X]$ be an irreducible polynomial. Then $F(p)|_fF$ is an algebraic extension. Moreover, if $\deg p=n+1$ then $[F(p):F]\le n+1$.
\end{teo}
\begin{proof}
Remember that $F(p)$ is generated by $\{1,\gamma,...,\gamma^n\}$ with $\gamma=\overline X$, $n\in\mathbb N$. Also, we can consider $n$ as the minimal integer such that there exist $a_0,...,a_{n+1}$ with 
$$0\in a_0+a_1\gamma+...+a_{n+1}\gamma^{n+1}.$$

Now let $b_0+b_1\gamma+...+b_n\gamma^n\in F(p)^\ast$. Since $x\cdot y\ne\emptyset$ for all $x,y\in F(P)$, for all $k=0,...,n$, there exist
$$d_{k0}+d_{k1}\gamma+...+d_{kn}\gamma^n\in(b_0+b_1\gamma+...+b_n\gamma^n)^k$$
for suitable $d_{ij}\in F$. Writing this in matrix notation, we have
\begin{align*}
  D\begin{pmatrix}
 1\\\gamma\\\vdots\\\gamma^n
 \end{pmatrix}\subseteq \begin{pmatrix}
 (b_0+b_1\gamma+...+b_n\gamma^n)^0\\(b_0+b_1\gamma+...+b_n\gamma^n)^1\\\vdots\\(b_0+b_1\gamma+...+b_n\gamma^n)^n
 \end{pmatrix}  
\end{align*}
 with
 $$D=\begin{pmatrix}d_{00} & d_{01} & \ldots & d_{1n} \\
 d_{10} & d_{11} & \ldots & d_{1n} \\
 \vdots & \vdots & \ddots & \vdots \\
 d_{n0} & d_{n1} & \ldots & d_{nn}
 \end{pmatrix}$$
 The fact that $F(p)$ is almost full enable us to scale the matrix $D$, saying
 $$D_{scaled}=
 \begin{pmatrix}e_{00} & e_{01} & \ldots & e_{1n} \\
 0 & e_{11} & \ldots & e_{1n} \\
 \vdots & \vdots & \ddots & \vdots \\
 0 & 0 & \ldots & e_{nn}
 \end{pmatrix}$$
 and getting
 \begin{align*}
   \tag{$\ast$}
   \begin{pmatrix}e_{00} & e_{01} & \ldots & e_{1n} \\
 0 & e_{11} & \ldots & e_{1n} \\
 \vdots & \vdots & \ddots & \vdots \\
 0 & 0 & \ldots & e_{nn}
 \end{pmatrix}
 \begin{pmatrix}
 1\\\gamma\\\vdots\\\gamma^n
 \end{pmatrix}\in 
 \begin{pmatrix}
 \sum^n_{j=0}g_{0j}(b_0+b_1\gamma+...+b_n\gamma^n)^j\\
 \sum^n_{j=0}g_{1j}(b_0+b_1\gamma+...+b_n\gamma^n)^j\\ \vdots\\
 \sum^n_{j=0}g_{nj}(b_0+b_1\gamma+...+b_n\gamma^n)^j
 \end{pmatrix}
 \end{align*}
 for suitable $g_{ij}\in F$.
 
 If $D_{scaled}$ is not invertible then $0\in\det(D_{scaled})=e_{11}e_{22}...e_{nn}$ and then, $e_{ii}=0$ for some $i\in\{1,...,n\}$ (see Lemma \ref{scal4}), which imply (by the very scalation process) that there exist a row $i$ with $L_i$ being a linear combination of the others. Suppose without loss of generality that 
$$L_{n+1}\cap\left[\left(\sum^{r_1}_{j=1}\lambda_{j1}\right)L_1+...+\left(\sum^{r_n}_{j=1}\lambda_{jn}\right)L_n\right]\ne\emptyset.$$
 This means
 $$0\in z_0+z_1(b_0+b_1\gamma+...+b_n\gamma^n)^1+...+z_n(b_0+b_1\gamma+...+b_n\gamma^n)^{n-1}-(b_0+b_1\gamma+...+b_n\gamma^n)^{n+1},$$
 for suitable $z_0,...,z_n\in F$, and then, for $f(X)=z_0+z_1X+...+z_nX^n-X^{n+1}$, we have 
 $$0\in f(b_0+b_1\gamma+...+b_n\gamma^n),$$
 which means $b_0+b_1\gamma+...+b_n\gamma^n$ is algebraic. If $D_{scaled}$ is invertible, since $F(p)|_fF$ is almost full we get
 $$\begin{pmatrix}
 1\\\gamma\\\vdots\\\gamma^n
 \end{pmatrix}
 \in D^{-1}_{scaled}\left[D_{scaled}\begin{pmatrix}
 1\\\gamma\\\vdots\\\gamma^n
 \end{pmatrix}\right].$$
 After multiplying the equation ($\ast$) by $D_{scaled}^{-1}$ we arrive at a system
 \begin{align*}
   \begin{pmatrix}
 1\\\gamma\\\vdots\\\gamma^n
 \end{pmatrix}
 \in D^{-1}_{scaled}\left[D_{scaled}\begin{pmatrix}
 1\\\gamma\\\vdots\\\gamma^n
 \end{pmatrix}\right]\subseteq D^{-1}_{scaled}
 \begin{pmatrix}
 \sum^n_{j=0}g_{0j}(b_0+b_1\gamma+...+b_n\gamma^n)^j\\
 \sum^n_{j=0}g_{1j}(b_0+b_1\gamma+...+b_n\gamma^n)^j\\ \vdots\\
 \sum^n_{j=0}g_{nj}(b_0+b_1\gamma+...+b_n\gamma^n)^j
 \end{pmatrix}
 \end{align*} 
 then our situation is
 \begin{align*}
   \tag{$\ast\ast$}\begin{pmatrix}
 1\\\gamma\\\vdots\\\gamma^n
 \end{pmatrix}
 \in D_{Scaled}^{-1}
 \begin{pmatrix}
 \sum^n_{j=0}g_{0j}(b_0+b_1\gamma+...+b_n\gamma^n)^j\\
 \sum^n_{j=0}g_{1j}(b_0+b_1\gamma+...+b_n\gamma^n)^j\\ \vdots\\
 \sum^n_{j=0}g_{nj}(b_0+b_1\gamma+...+b_n\gamma^n)^j
 \end{pmatrix}
 \end{align*}
 Let $D_{Scaled}^{-1}=(h_{ij})$. From ($\ast\ast$), after calculating the matrix product we get
 $$\begin{cases}
     \gamma^0\in \sum^n_{j=0}g_{0j}h_{0j}(b_0+b_1\gamma+...+b_n\gamma^n)^j\\
     \gamma_1\in \sum^n_{j=0}g_{1j}h_{1j}(b_0+b_1\gamma+...+b_n\gamma^n)^j\\
     \vdots\\
     \gamma^n\in\sum^n_{j=0}g_{nj}h_{nj}(b_0+b_1\gamma+...+b_n\gamma^n)^j
 \end{cases}$$
which imply
 $$\begin{cases}
     a_0\gamma^0\in \sum^n_{j=0}a_0g_{0j}h_{0j}(b_0+b_1\gamma+...+b_n\gamma^n)^j\\
     a_1\gamma_1\in \sum^n_{j=0}a_1g_{1j}h_{1j}(b_0+b_1\gamma+...+b_n\gamma^n)^j\\
     \vdots\\
     a_n\gamma^n\in\sum^n_{j=0}a_ng_{nj}h_{nj}(b_0+b_1\gamma+...+b_n\gamma^n)^j
 \end{cases}$$

 Then
 \begin{align*}
   a_1a_n\gamma^{n+1}&\subseteq
 \left(\sum^n_{j=0}a_1g_{1j}h_{1j}(b_0+b_1\gamma+...+b_n\gamma^n)^j\right)
 \left(\sum^n_{j=0}a_ng_{nj}h_{nj}(b_0+b_1\gamma+...+b_n\gamma^n)^j\right)
 \end{align*}
 and
 \begin{align*}
   &0\in a_0+a_1\gamma+...+a_{n+1}\gamma^{n+1}\subseteq\\
   &\sum^n_{p=0}\left[\sum^n_{j=0}a_pg_{pj}h_{pj}(b_0+b_1\gamma+...+b_n\gamma^n)^j\right]+\\
   &\left(\sum^n_{j=0}a_1g_{1j}h_{1j}(b_0+b_1\gamma+...+b_n\gamma^n)^j\right)
 \left(\sum^n_{j=0}a_ng_{nj}h_{nj}(b_0+b_1\gamma+...+b_n\gamma^n)^j\right).
 \end{align*}

 Now, thinking with polynomials, we have
 \begin{align*}
   A(X)&:=\sum^n_{p=0} \left[\sum^n_{j=0}a_pg_{pj}h_{pj}X^j\right]+\left(\sum^n_{j=0}a_1g_{1j}h_{1j}X^j\right)
 \left(\sum^n_{j=0}a_ng_{nj}h_{nj}X^j\right) =\\
&\left[\sum^n_{p=0}\sum^n_{j=0}a_pg_{pj}h_{pj}\right]X^j+\left(\sum^n_{j=0}a_1g_{1j}h_{1j}X^j\right)
 \left(\sum^n_{j=0}a_ng_{nj}h_{nj}X^j\right)=\\
 &=P(X)+S(X)T(X),
 \end{align*}
 with
 \begin{align*}
P(X)&=\left[\sum^n_{p=0}\sum^n_{j=0}a_pg_{pj}h_{pj}\right]X^j\\
S(X)&=\sum^n_{j=0}a_1g_{1j}h_{1j}X^j\\
T(X)&=\sum^n_{j=0}a_ng_{nj}h_{nj}X^j
 \end{align*}
 Then
 $$0\in ev\left(A(X),b_0+b_1\gamma+...+b_n\gamma^n\right);$$
 which means that there exists at least a polynomial $f(X)\in A(X)=P(X)+S(X)T(X)$ with 
 $$0\in f(b_0+b_1\gamma+...+b_n\gamma^n).$$
 Then $b_0+b_1\gamma+...+b_n\gamma^n$ is algebraic. Of course, this also imply that $[F(p):F]\le n+1$.
\end{proof}

\section{Vector Spaces}

Since we already have available matrices and polynomials for superrings, a natural extension for the theory is a sort of ''vector space'' and some linear algebra methods. We start this program here, proceeding in a very similar fashion of Hofmann's and Kunze's Linear Algebra Book (\cite{hoffmanlinear}).

\begin{defn}\label{mvec}
    A \textbf{(multi) vector space} over a superfield $F$ is a tuple $(V,+,\cdot,0)$ such that $(V,+,0)$ is an abelian multigroup and $\cdot:F\times V\rightarrow\mathcal P^\ast(V)$ is a function (which image denoted by $\cdot(\lambda,v):=\lambda v$) satisfying the following properties for all $\lambda,\mu\in F$ and all $v,w\in V$:
    \begin{description}
        \item [MV0 -] $1v=\{v\}$ and $0\cdot v=\{0\}$;
        \item [MV1 -] $(\lambda\mu)v=\lambda(\mu v)$.
    \end{description}
    Here we adopt the following convention: if $A\subseteq F$ and $v\in V$, we set
    $$Av:=\bigcup\{\lambda v:\lambda\in A\}.$$
    \begin{description}
        \item [MV2 -] $\lambda(v+w)\subseteq\lambda v+\lambda w$;
        \item [MV3 -] $(\lambda+\mu)v\subseteq\lambda v+\mu v$.
    \end{description}
    The vector space $(V,+,0)$ is \textbf{full} if the equality holds in MV2 and MV3.
\end{defn}

We proceed similarly to the practice used with polynomials and matrices: we omit the word ''multi'' and just say ''vector spaces'' over superfields.

Of course, we stick to vectors spaces here but it is available the Definition for ''modules'', just repacling superfields in Definition \ref{mvec} for superring.

Here are some natural examples of vector spaces.

\begin{prop}\label{extvec}
    Let $K|F$ be a superfield extension. Then $K$ is a $F$-vector space, which is full iff the extension is full.
\end{prop}
\begin{proof}
    Here $\cdot:F\rightarrow K\rightarrow\mathcal P^\ast (K)$ is just the restriction of multiplication to $F$ on the first coordinate. M0 is immediate and M1-M3 are consequences of the axioms of superrings. It is immediat that $K$ is a full vector space iff $K|_fF$.
\end{proof}

\begin{teo}\label{fnvec}
    Let $F^n$ be the usual $n$-folded cartesian product $F\times...\times F$. We already know that $F^n$ with the induced sum is a multigroup. Now, for $\lambda\in F$ and $v=(x_1,...,x_n)\in F^n$ define
    $$\lambda v:=(\lambda x_1,...,\lambda x_n).$$
    Then $(F^n,+,\cdot,0)$ is a vector space. Moreover $F^n$ is full iff $F$ is full.
\end{teo}
\begin{proof}
    We already have that $F^n$ is commutative a superring. By the very Definition os scalar product we get $1v=v$. Now let $v,w\in F^n$, $v=(x_1,...,x_n)$, $w=(y_1,...,y_n)$ and $\lambda,\mu\in F$. We have
    \begin{align*}
        (\lambda+\mu)&v:=((\lambda+\mu)x_1,...,(\lambda+\mu)x_n)\\
        &\subseteq(\lambda x_1+\mu x_1,...,\lambda x_n+\mu x_n) \\
        &=(\lambda x_1,...,\lambda x_n)+(\mu x_1,...,\mu x_n)=\lambda v+\mu w.
    \end{align*}
    Similarly we conclude that $(\lambda+\mu)v\subseteq\lambda v+\mu v$.

    Then $F^n$ is a vector space which is full if $F$ is full.

    Now suppose $F^n$ full. Then for $\alpha,\lambda,\mu\in F$ and $v=(\alpha,0,...,0)$ we have 
    $$(\lambda\alpha+\mu\alpha,0,...,0)=\lambda v+\lambda v=(\lambda+\mu)v=((\lambda+\mu)\alpha,0,...,0);$$
    which means $(\lambda+\mu)\alpha=\lambda\alpha+\mu\alpha$. Similarly we conclude that $\alpha(\lambda+\mu)=\alpha\lambda+\alpha\mu$.
\end{proof}

\begin{teo}
    Let $F$ be a superfield and $n\ge1$. Then $M_n(F)$ is a vector space which is full iff $F$ is full.
\end{teo}
\begin{proof}
    This is consequence of Lemma \ref{matrix4}, identifying $F$ with $M_{1\times1}(F)$.
\end{proof}

\begin{teo}
    Let $F$ be a superfield and $n\ge1$. Then $F[X_1,...,X_n]$ is a vector space which is full iff $F$ is full.
\end{teo}
\begin{proof}
    The argument here is similar to the one in Theorem \ref{fnvec}.
\end{proof}

\begin{defn}[Subspace]
    Let $V$ be a $F$-vector space and $W\subseteq V$. We say that $W$ is a \textbf{subespace} if $0\in W$ and for all $w_1,w_2\in W$ and all $\lambda\in F$ we have $w_1+w_2\subseteq W$ and $\lambda w_1\subseteq W$.
\end{defn}

\begin{teo}
    Let $F$ be a full superfield and consider a system $Ax=0$, $A\in M_{n\times m}(F)$. Then
    $$\mbox{Sol}[Ax=0]:=\{v\in M_{n\times 1}(F):0\in Av\}$$
    is a subspace of $M_{n\times 1}(F)$.
\end{teo}
\begin{proof}
    This is another consequence of Lemma \ref{matrix4}. We need $F$ full in order to conclude that if $0\in Av$ and $0\in Aw$ then $0\in A(v+w)=Av+Aw$.
\end{proof}

\begin{defn}[Spanned Subspace]
    Let $V$ be a $F$-vector space and $A\subseteq V$. The \textbf{subspace generated by $A$} is defined by
    $$\langle A\rangle:=\bigcap\{W\subseteq V:W\mbox{ is a subspace and }A\subseteq W\}.$$
\end{defn}



\begin{defn}[Linear Combination]
    Let $V$ be a $F$-vector space, $A\subseteq V$ and $w\in V$. We say that $w$ is a \textbf{linear combination} of elements in $A$ if there exist $\{v_1,...,v_n\}\subseteq A$ with $$w\in\sum^{r_1}_{j=1}\lambda_{j1}v_1+...+\sum^{r_n}_{j=1}\lambda_{jn}v_n$$
    for some $\lambda_{ij}\in F$. We denote the set of linear combinations of $V$ by
    \begin{align*}
        \mathcal{CL}(A)&=\bigcup\left\lbrace\sum^{r_1}_{j=1}\lambda_{j1}v_1+...+\sum^{r_n}_{j=1}\lambda_{jn}v_n:\{v_1,...,v_n\}\subseteq A,\,\lambda_{ij}\in F,\,r_1,...,r_n\in\mathbb N\right\rbrace.
    \end{align*}
\end{defn}

If $V$ is full, then 
$$\mathcal{CL}(A):=\bigcup\{\lambda_1v_1+...+\lambda _nv_n:v_i\in A,\,\lambda_i\in F,\,i=1,...,n,\,n\ge1\}.$$

\begin{teo}\label{gen1}
    Let $V$ be a $F$-vector space and $A\subseteq V$. Then $\langle A\rangle=\mathcal{CL}(A)$.
\end{teo}
\begin{proof}
    We have that $\mathcal{CL}(A)$ is a subspace, which provide $\langle A\rangle\subseteq\mathcal{CL}(A)$. If $W\subseteq V$ and $A\subseteq W$, by the very Definition of subspace (and induction) we have $\mathcal{CL}(A)\subseteq W$, which provide $\mathcal{CL}(A)\subseteq\langle A\rangle$.
\end{proof}

\begin{lem}\label{gen2}
    Let $V$ be a $F$-vector space and $A,B\subseteq V$. Then
    \begin{enumerate}[i -]
        \item $\langle\langle A\rangle\rangle=\langle A\rangle$;
        \item if $A\subseteq B$ then $\langle A\rangle$ is a subspace of $\langle B\rangle$;
        \item if $A\subseteq B$ and for all $v\in B$, $v\in\langle A\rangle$ then $\langle A\rangle=\langle B\rangle$.
    \end{enumerate}
\end{lem}
\begin{proof}
    
\end{proof}

\begin{defn}[Linear Independence]
    Let $V$ be a $F$-vector space and $A\subseteq V$. We say that $A$ is \textbf{$F$-linearly independent} if for all 
distinct $v_1,...,v_n\in A$, $n\in\mathbb N$, the following hold:
$$\mbox{If }0\in\sum^{r_1}_{j=1}\lambda_{j1}v_1+...+\sum^{r_n}_{j=1}\lambda_{jn}v_n\mbox{ then }0\in\sum^{r_1}_{j=1}\lambda_{ji}\mbox{ for all }i=1,...,n.$$
and $I$ is \textbf{$F$-linearly dependent} if it is not $F$-linearly independent.
\end{defn}

\begin{defn}[Base]
    Let $V$ be a $F$-vector space and $B\subseteq V$. We say that $B$ is a \textbf{$F$-basis} if $B$ is linear independent and $V=\langle B\rangle$.
\end{defn}

\begin{defn}
    We say that a $F$-vector space is \textbf{finitely generated} if $V=\langle S\rangle$ for some $S\subseteq V$ finite.
\end{defn}

\begin{teo}\label{basis1}
    Every finitely generated $F$-vector space $V$ has a basis.
\end{teo}
\begin{proof}
    Let $V$ be a finitely generated $F$-vector space with $V=\langle v_1,...,v_n\rangle$ ($v_1,...,v_n\in V$). If $\{v_1,...,v_n\}$ is LI we are done. If not, after a rearrangement of indexes if necessary, we can suppose without loss of generality that $v_1\in\mathcal{CL}(\{v_2,...,v_n\})$. Then (using Theorem \ref{gen1} and Lemma \ref{gen2}) we have
    $$V=\mathcal{CL}(\{v_1,...,v_n\})=\mathcal{CL}(\{v_2,...,v_n\}).$$
    If $\{v_2,...,v_n\}$ we are done. If not, suppose without loss of generality that $v_2\in\mathcal{CL}(\{v_3,...,v_n\})$. Then we have
    $$V=\mathcal{CL}(\{v_1,...,v_n\})=\mathcal{CL}(\{v_2,...,v_n\})=\mathcal{CL}(\{v_3,...,v_n\}).$$
    Repeating this process, after a number finite of steps we arrive at a basis $\{v_k,v_{k+1},...,v_n\}$ of $V$ for some $k$ with $1\le k\le n$.
\end{proof}

Unfortunately, we do not know if, for general superfields $F$, all basis in a finitely generated $F$-vector spaces has the same dimension. In order to deal with this question, we propose the following concept.

\begin{defn}\label{linearly-closed}
    Let $F$ be superfield. We say that $F$ is \textbf{linearly closed} if the system $Ax=0$ has at least a non trivial solution weak solution for all $A\in M_{n\times m}(F)$ with $m>n$.
\end{defn}

Of course, every field is a linearly closed superfield. As we will see later (Theorem \ref{basis5}), this is also the case for hyperfields. The concept of linearly closeness is useful to get the notion of dimension for a subclass of finitely generated $F$-vector spaces.

\begin{teo}\label{basis2}
    Let $F$ be a linearly closed superfield and $V$ be a finitely generated $F$-vector space with $V=\langle v_1,...,v_n\rangle$ ($v_1,...,v_n\in V$). Then every linear independent subset of $V$ has at most $n$ elements.
\end{teo}
\begin{proof}
    We just need to prove that if $S\subseteq V$ and $|S|>n$ then $S$ is linearly dependent.

    Let $S$ be such set with $S=\{w_1,...,w_m\}$, $m>n$. Since $V=\langle v_1,...,v_n\rangle$, then there exists scalars $a_{ij}\in F$ with
    $$w_j\in a_{1j}v_1+...+a_{nj}v_n,\,j=1,...,m.$$
    Then for all $\lambda_1,...,\lambda_m\in F$ we get
    \begin{align*}        \lambda_1w_1+...+\lambda_mw_m&=\sum^m_{j=1}\lambda_jw_j         \subseteq\sum^m_{j=1}\lambda_j\sum^n_{i=1}a_{ij}v_i\subseteq\sum^m_{j=1}\sum^n_{i=1}(\lambda_ja_{ij})v_i=\sum^n_{i=1}\left[\sum^m_{j=1}\lambda_ja_{ij}\right]v_i
    \end{align*}
    If $0\in\lambda_1w_1+...+\lambda_mw_m$ then we get
    $$0\in\sum^n_{i=1}\left[\sum^m_{j=1}\lambda_ja_{ij}\right]v_i,$$
    providing
    $$0\in\begin{pmatrix}a_{11} & a_{21} & \ldots & a_{m1} \\
 a_{21} & a_{22} & \ldots & a_{m2} \\
 \vdots & \vdots & \ddots & \vdots \\
 a_{1n} & a_{2n} & \ldots & a_{mn}
 \end{pmatrix}\begin{pmatrix}\lambda_1\\ \lambda_2\\ \vdots \\ \lambda_m\end{pmatrix}$$
 Let
 $$A=\begin{pmatrix}a_{11} & a_{21} & \ldots & a_{m1} \\
 a_{21} & a_{22} & \ldots & a_{m2} \\
 \vdots & \vdots & \ddots & \vdots \\
 a_{1n} & a_{2n} & \ldots & a_{mn}
 \end{pmatrix}$$
 Since $F$ is linearly closed,  the system $Ax=0$ has a weak solution if $m>n$, and we have that $S$ is linear dependent if $m>n$.
\end{proof}



    

\begin{teo}\label{basis4}
    Let $F$ be a linearly closed superfield and $V$ be a $F$-vector space finitely generated. If $B_1$ and $B_2$ are basis of $V$ then $|B_1|=|B_2|$.
\end{teo}
\begin{proof}
    Let $B_1=\{v_1,...,v_n\}$ and $B_2=\{w_1,...,w_m\}$. Since $V=\langle B_1\rangle$ and $B_2$ is linearly independent, by Theorem \ref{basis2} we get $m\le n$. Since $V=\langle B_2\rangle$ and $B_1$ is linearly independent, by Theorem \ref{basis2} we get $n\le m$. Then $m=n$.
\end{proof}

\begin{defn}
    Let $F$ be a linearly closed superfield and $V$ be a $F$-vector space finitely generated. We define the \textbf{dimension} of $V$ by $\dim(V):=|B|$ where $B\subseteq V$ is any basis of $V$.
\end{defn}

\begin{teo}\label{basis5}
    Every hyperfield is linearly closed as a superfield.
\end{teo}
\begin{proof}
    First, remember that every hyperfield is full. Now, let $A\in M_{n\times m}(F)$ with $m>n$. Second, since $F$ is a hyperfield, considering the vector space $F^n$, we have a full $F$-vector space such that for all $\lambda\in F$ and all $v\in F^n$, $\lambda v$ is a singleton set.
    
    Keeping this in mind, we will construct a weak solution $b\in M_{m\times 1}(F)$ dividing the proof in some cases.

    \textbf{Case I - $n=1$ (which imply $m\ge2$)}. In this case, let $A\in M_{m\times 1}(F)$ with
    $$A=\begin{pmatrix}
        a_1 & a_2 & ... & a_m
    \end{pmatrix}$$
    We need to find $x_1,...,x_m\in F$ (not all zero) such that
    $$0\in a_1x_1+a_2x_2+...+a_mx_m.$$
    If $a_1=a_2=0$, just choose $x_1=x_2=1$ and $x_i=0$ for all $i\ge3$. If $0\notin\{a_1,a_2\}$, choose $x_i=0$ for all $i\ge3$, $x_2=1$ and $x_1=-a_1^{-1}a_2$. Then
    $$a_1x_1+a_2x_2+a_3x_3+...+a_mx_m=a_1x_1+a_2x_2=a_1[-a_1^{-1}a_2]+a_2=a_2-a_2\mbox{ with }0\in a_2-a_2.$$
    We further refer to this tactic to find $x_1,x_2,...,x_m$ as ''the Case I method''\label{case-I-method}.
    
    $$\qquad\qquad\qquad\ast\ast\ast\qquad\qquad\qquad$$
    
    \textbf{Case II - $n=2$ (which imply $m\ge3$)}. In this case, let $A\in M_{m\times 2}(F)$ with
    $$A=\begin{pmatrix}
        a_1 & a_2 & ... & a_m \\
        b_1 & b_2 & ... & b_m
    \end{pmatrix}$$
    We need to find $x_1,...,x_m\in F$ (not all zero) such that
    \begin{align}\label{eq-01-sys-case-II}
        0&\in a_1x_1+a_2x_2+...+a_mx_m\nonumber \\
        0&\in b_1x_1+b_2x_2+...+b_mx_m
    \end{align}
    We have some cases here. If $a_j=b_j=0$ for some $j$, just choose $x_j=1$ and $x_i=0$ for all $i\ne j$. If $0\in\{a_1,...,a_m,b_1,...,b_m\}$, saying $b_1=0$ and $a_1\ne0$. Then we are reduced to
    \begin{align}\label{eq-02-sys-case-II}
        0&\in a_1x_1+a_2x_2+...+a_mx_m\nonumber\\
        0&\in b_2x_2+...+b_mx_m
    \end{align}
    Now find some $d_2,...,d_m\in F$ (not all zero) with $0\in b_2d_2+...+b_md_m$ (as in the Case I method) and choose $d_1$ in order to get $d_1\in -a_1^{-1}[a_2d_2+...+a_md_m]$. So $(d_1,...,d_m)$ is a non trivial weak solution of the reduced system \ref{eq-02-sys-case-II}. Now, let $0\notin\{a_1,...,a_m,b_1,...,b_m\}$. If $\{(a_1,...,a_m),(b_1,...,b_m)\}$ is LD, saying $(b_1,...,b_m)=\lambda(a_1,...,a_m)$, we just need to choose $d_1,d_2,...,d_m\in F$ not all zero (as in case I) such that $0\in a_1d_1+a_2d_2+...+a_md_m$ in order to get a solution of \ref{eq-01-sys-case-II} for this case.

    Now, suppose $0\notin\{a_1,...,a_m,b_1,...,b_m\}$ and $\{(a_1,...,a_m),(b_1,...,b_m)\}$ LI. Since $F$ is a hyperfield, to find a non trivial weak solution of \ref{eq-01-sys-case-II} is equivalent to find a non trivial weak solution of
    \begin{align*}
        0&\in x_1+a_1^{-1}a_2x_2+...+a_1^{-1}a_mx_m\nonumber \\
        0&\in x_1+b_1^{-1}b_2x_2+...+b_1^{-1}b_mx_m
    \end{align*}
    Then we can suppose without loss of generality that $a_1=b_1=1$, and we need to find a weak solution of
    \begin{align}\label{eq-03-sys-case-II}
        0&\in x_1+a_2x_2+...+a_mx_m\nonumber \\
        0&\in x_1+b_2x_2+...+b_mx_m
    \end{align}
    Now, consider the set of systems obtained after the elementary operation $L_2\leftarrow L_2-L_1$:
    \begin{align}\label{eq-04-sys-case-II}
        0&\in x_1+a_2x_2+...+a_mx_m\nonumber \\
        0&\in (1-1)x_1+(b_2-a_2)x_2+...+(b_m-a_m)x_m
    \end{align}
    As in Case I, let $d_2,...,d_m\in F$ (not all zero) with
    $$0\in (b_2-a_2)d_2+...+(b_m-a_m)d_m.$$
    In particular, there exist $z\in F$ with
    $$z\in(a_2x_2+...+a_mx_m)\cap(a_2x_2+...+a_mx_m).$$
    Let $x_1=-z$. Then $(-z,d_2,d_3,...,d_m)$ is a weak solution of both \ref{eq-03-sys-case-II} and \ref{eq-04-sys-case-II}, completing the proof for Case II. We further refer to this tactic to find $z$ as ''the Case II method''\label{case-II-method}.
    
    $$\qquad\qquad\qquad\ast\ast\ast\qquad\qquad\qquad$$

    \textbf{Case III - $n=3$ (which imply $m\ge4$)}. In this case, let $A\in M_{m\times 2}(F)$ with
    $$A=\begin{pmatrix}
        a_1 & a_2 & ... & a_m \\
        b_1 & b_2 & ... & b_m \\
        c_1 & c_2 & ... & c_m
    \end{pmatrix}$$
    We need to find $x_1,...,x_m\in F$ (not all zero) such that
    \begin{align}\label{eq-01-sys-case-III}
        0&\in a_1x_1+a_2x_2+...+a_mx_m\nonumber \\
        0&\in b_1x_1+b_2x_2+...+b_mx_m \nonumber \\
        0&\in c_1x_1+c_2x_2+...+c_mx_m
    \end{align}
    We cases to deal with. If $a_j=b_j=c_j=0$ for some $j$, just choose $x_j=1$ and $x_i=0$ for all $i\ne j$. Then we can suppose without loss of generality that $0\notin\{a_1,b_1,c_1\}$. Choosing $x_j=0$ for $j\ge5$, we only need to deal with the reduced system
    \begin{align}\label{eq-02-sys-case-III}
        0&\in a_1x_1+a_2x_2+a_3x_3+a_4x_4\nonumber \\
        0&\in b_1x_1+b_2x_2+b_3x_3+b_4x_4 \nonumber \\
        0&\in c_1x_1+c_2x_2+c_3x_3+c_4x_4
    \end{align}
    Also, Since $F$ is a hyperfield, to find a non trivial weak solution of \ref{eq-02-sys-case-III} is equivalent to find a non trivial weak solution of
    \begin{align*}
        0&\in x_1+a_1^{-1}a_2x_2+a_1^{-1}a_3x_3+a_1^{-1}a_4x_4\nonumber \\
        0&\in x_1+b_1^{-1}b_2x_2+b_1^{-1}b_3x_3+b_1^{-1}b_4x_4 \nonumber \\
        0&\in x_1+c_1^{-1}c_2x_2+c_1^{-1}c_3x_3+c_1^{-1}c_4x_4
    \end{align*}
    Then we can suppose without loss of generality that $a_1=b_1=c_1=1$ and only deal with the new reduced system
    \begin{align}\label{eq-03-sys-case-III}
        0&\in x_1+a_2x_2+a_3x_3+a_4x_4\nonumber \\
        0&\in x_1+b_2x_2+b_3x_3+b_4x_4 \nonumber \\
        0&\in x_1+c_2x_2+c_3x_3+c_4x_4
    \end{align}
    If one of the sets $\{(1,a_2,a_3,a_4);(1,b_2,b_3,b_4)\}$, $\{(1,a_2,a_3,a_4);(1,c_2,c_3,c_4)\}$ and $\{(1,b_2,b_3,b_4);(1,c_2,c_3,c_4)\}$ is LD, we are reduced to case II. Then we can suppose that all these three sets are LI (which imply, in particular that $0\notin b_i-a_i$ and $0\notin b_i-a_i$ for $i=2,3,4$). Performing the elementary operations $\{L_2\leftarrow L_2-L_1;L_3\leftarrow L_3-L_1\}$ we arrive at the systems
    \begin{align*}
        0&\in x_1+a_2x_2+a_3x_3+a_4x_4\nonumber \\
        0&\in (1-1)x_1+(b_2-a_2)x_2+(b_3-a_3)x_3+(b_4-a_4)x_4 \nonumber \\
        0&\in (1-1)x_1+(c_2-a_2)x_2+(c_3-a_3)x_3+(c_4-a_4)x_4
    \end{align*}
    Consider the scaled system
    \begin{align}\label{eq-04-sys-case-III}
        0&\in x_1+a_2x_2+a_3x_3+a_4x_4\nonumber \\
        0&\in (b_2-a_2)x_2+(b_3-a_3)x_3+(b_4-a_4)x_4 \nonumber \\
        0&\in (c_2-a_2)x_2+(c_3-a_3)x_3+(c_4-a_4)x_4
    \end{align}
    If $b_i-a_i=\lambda(c_i-a_i)$ for all $i=2,3,4$ or $b_i-a_i=\lambda(c_i-a_i)$ for all $i=2,3,4$ we are reduced to the Case II method (\ref{case-II-method}). If not, since $0\notin b_2-a_2$ and $0\notin c_2-a_2$, performing the elementary operations $\{L_3\leftarrow (b_2-a_2)L_3-(c_2-a_2)L_2\}$ we arrive at the systems
    \begin{align*}
        0&\in x_1+a_2x_2+a_3x_3+a_4x_4\nonumber \\
        0&\in (b_2-a_2)x_2+(b_3-a_3)x_3+(b_4-a_4)x_4 \nonumber \\
        0&\in [(b_2-a_2)(c_2-a_2)-(c_2-a_2)(b_2-a_2)]x_2+\nonumber\\
        &+[(b_2-a_2)(c_3-a_3)-(c_2-a_2)(b_3-a_3)]x_3+[(b_2-a_2)(c_4-a_4)-(c_2-a_2)(b_4-a_4)]x_4
    \end{align*}
    Consider the (second) scaled system
    \begin{align}\label{eq-05-sys-case-III}
        0&\in x_1+a_2x_2+a_3x_3+a_4x_4\nonumber \\
        0&\in (b_2-a_2)x_2+(b_3-a_3)x_3+(b_4-a_4)x_4 \nonumber \\
        0&\in [(b_2-a_2)(c_2-a_2)-(c_2-a_2)(b_3-a_3)]x_3+[(b_2-a_2)(c_4-a_4)-(c_2-a_2)(b_4-a_4)]x_4
    \end{align}
    Using the Case I Method (\ref{case-I-method}) we can find $d_3,d_4\in F$ (not all zero) with
    $$0\in [(b_2-a_2)(c_3-a_3)-(c_2-a_2)(b_3-a_3)]d_3+[(b_2-a_2)(c_4-a_4)-(c_2-a_2)(b_4-a_4)]d_4,$$
    which imply in particular that there exist some $z\in F$ with
    $$z\in[(b_2-a_2)(c_2-a_2)d_3+(b_2-a_2)(c_4-a_4)d_4]\cap[(c_2-a_2)(b_3-a_3)d_3+(c_2-a_2)(b_4-a_4)d_4].$$
    \kmar{Let $d_2\in -[(b_3-a_3)d_3+(b_4-a_4)d_4]$ such that $-z\in(c_2-a_2)d_2$}. Then
    $$0\in(c_2-a_2)[(b_2-a_2)d_2+(b_3-a_3)d_3+(b_4-a_4)d_4],$$
    which imply
    $$0\in(b_2-a_2)d_2+(b_3-a_3)d_3+(b_4-a_4)d_4,$$
    and in particular, there exist
    $$w\in[a_2d_2+a_3d_3+a_4d_4]\cap[b_2d_2+b_3d_3+b_4d_4].$$
    Picking $d_1=-w$, we have that $(d_1,d_2,d_3,d_4)$ is a weak solution of both (set of) systems \ref{eq-03-sys-case-III}, \ref{eq-04-sys-case-III} and \ref{eq-05-sys-case-III} completing the proof for Case III. We further refer to this tactic to find $z$ and $w$ as ''the Case III method''\label{case-III-method}.
    
    $$\qquad\qquad\qquad\ast\ast\ast\qquad\qquad\qquad$$

    \textbf{Case IV - the general case $m>n$ (and $n\ge3$)}. Just proceed by induction on $m$. The base cases are Case I and II and the induction step is an argument similar to the the Case III method (\ref{case-III-method}).
\end{proof}

As we can see, we had a lot of effort in order to prove Theorem \ref{basis5}. In this sense, we propose the following questions.

\begin{op}
$ $
\begin{enumerate}
    \item Is every full superfield $F$ a linearly closed superfield?
    \item What are the necessary conditions for a superfield $F$ be a linearly closed one?
\end{enumerate}
\end{op}

As application of these fragment of linear algebra for superfields, we get the following Theorem, which is a consequence of combining Theorem \ref{teohell2}, Proposition \ref{extvec}, Theorem \ref{basis2} and Theorem \ref{basis5}.

\begin{teo}
    Let $F$ be a linearly closed superfield and $p\in F[X]$ be an irreducible polynomial with $\deg p=n+1$. Then $F(p)$ is a full $F$-vector space and $\dim(F(p))=n+1$.
\end{teo}

\begin{teo}
    Let $F$ is a linearly closed superfield and $p\in F[X]$ be an irreducible polynomial with $\deg p=n+1$. Then $F(p)$ is also linearly closed.
\end{teo}
\begin{proof}
    Remember that $F(p)$ is generated by $\{1,\gamma,...,\gamma^n\}$ with $\gamma=\overline X$, $n\in\mathbb N$. Also, we can consider $n$ as the minimal integer such that there exist $a_0,...,a_{n+1}$ with 
    $$0\in d_0+d_1\gamma+...+d_{n+1}\gamma^{n+1}.$$
    
    Let $A\in M_{m\times q}(F(p))$, saying, $A=(\alpha_{ij})$. We can write each $\alpha_{ij}$ as
    $$\alpha_{ij}=a_{0ij}+a_{1ij}\gamma+...+a_{nij}\gamma^n$$
    for suitable $a_{kij}\in F$. Then a system $Ax=0$ over $F(p)$ can be split into $n+1$ systems $A_kx=0$ over $F$, where
    $A_k=(a_{kij})$ for each $k=0,1,...,n$ (in fact, $Ax=0$ means $A_0x+\gamma A_1x+\gamma^2A_2x+...+\gamma^nA_nx=0$). Since $F$ is linearly closed, each $A_kx=0$ has at least a non-trivial solution, providing a non-trivial solution for $Ax=0$.
\end{proof}

\section{Final remarks and and future works}

 We just have started  a development of a (or a possible approach to) ``multivalued Linear Algebra''. No doubt that will be necessary much effort to provide a more complete analysis of the notion of vector spaces, and define and analyze invariants subspaces. 
 
 The  fragment  of this multivalued Linear Algebra here development suggests that it may be a promising theory: the ``Kronecker construction'' (Theorem \ref{teohell2}) providing simple algebraic extensions in the superfield setting is a keystone to develop a program of studying the hyperfields and  superfields under a natural notion of algebraic extension and roots of polynomials. This program shares some common features with the recent work in \cite{baker2021descartes}. In fact,   in  \cite{roberto2022ACmultifields2}, based on Theorem \ref{teohell2} we show that every ``well behaved'' superfield has a (unique up to isomorphism) full algebraic extension to a superfield that is algebraically closed. 

 The next steps in the program of study algebraic extensions of superfields are a development of Galois theory and Galois cohomology theory,  envisaging application to other mathematical theories as abstract structures of quadratic forms and real algebraic geometry ((\cite{ribeiro2016functorial},\cite{roberto2021quadratic},\cite{roberto2021ktheory}). Some  parts of this program are under development in \cite{roberto2021hauptsatz} and \cite{roberto2022galois}; in particular, in \cite{roberto2021hauptsatz} we have developed a fragment of theory of quadratic extensions of hyperfields and superfields and, applying that to the theory of Special Groups (\cite{dickmann2000special}), we obtained a generalization of Arason-Pfister property (\cite{arason1971hauptsatz},\cite{dickmann2000special}) to {\em every} special groups.

\bibliographystyle{plain}
\bibliography{one_for_all}
\end{document}